\documentclass[11pt,abstract]{scrartcl}
\usepackage{amsmath}
\usepackage{amssymb}
\usepackage{amsthm}
\usepackage{physics}
\usepackage{units}
\usepackage[final]{microtype}
\usepackage{flafter}
\usepackage{caption}
\usepackage{subcaption}
\usepackage{booktabs}
\usepackage{datatool}
\usepackage{multirow}
\usepackage{xcolor}
\usepackage{tikz}
\usetikzlibrary{arrows.meta, external, calc}
\usepackage{pgfplots}
\pgfplotsset{compat=newest}
\usepackage{aliascnt} 			%
\usepackage[backend=biber,hyperref,backref,bibencoding=utf8,defernumbers=true,style=numeric-comp,giveninits=true,doi=false,isbn=false]{biblatex}
\renewbibmacro{in:}{\ifentrytype{article}{}{\printtext{\bibstring{in}\intitlepunct}}}
\usepackage{csquotes}
\usepackage{hyperref}

\newtheoremstyle{mytheorem}
	{}
	{}
	{\itshape}
	{}
	{\bfseries}
	{.}
	{ }
	{}
\newtheoremstyle{mydef}
	{}
	{}
	{}
	{}
	{\bfseries}
	{.}
	{ }
	{}
\theoremstyle{mytheorem}
	\newtheorem{theorem}{Theorem}
	\newaliascnt{lemma}{theorem}
	\newtheorem{lemma}[lemma]{Lemma}
	\aliascntresetthe{lemma}
	
	\newaliascnt{proposition}{theorem}
	
	\aliascntresetthe{proposition}
	
	\newaliascnt{corollary}{theorem}
	
	\aliascntresetthe{corollary}
	
\theoremstyle{mydef}
	\newaliascnt{definition}{theorem}
	\newtheorem{definition}[definition]{Definition}
	\aliascntresetthe{definition}
		
	\newaliascnt{example}{theorem}
	\newtheorem{example}[example]{Example}
	\aliascntresetthe{example}
		
	\newaliascnt{remark}{theorem}
	\newtheorem{remark}[remark]{Remark}
	\aliascntresetthe{remark}
	
	\newaliascnt{assumption}{theorem}
	
	\aliascntresetthe{assumption}

\renewcommand{\vec}[1]{{\mathbf{#1}}}
\newcommand{\hvec}[1]{\hat{{\mathbf{#1}}}}
\newcommand{\bvec}[1]{\bar{{\mathbf{#1}}}}
\newcommand{\jump}[1]{[\![{#1}]\!]}
\newcommand{\avg}[1]{\{\!\!\{{#1}\}\!\!\} }

\newcommand{\divt}{{\operatorname{div}\,}}
\newcommand{\divhat}{\widehat{\operatorname{div}}\,}
\newcommand{\divtilde}{\widetilde{\operatorname{div}}\,}
\newcommand{\divbar}{\overline{\operatorname{div}}\,}

\newcommand{\diam}{{\operatorname{diam}\,}}

\newcommand{\R}{{\mathbb R}}
\newcommand{\C}{{\mathbb C}}
\newcommand{\Q}{{\mathbb Q}}
\newcommand{\Z}{{\mathbb Z}}
\newcommand{\N}{{\mathbb N}}
\renewcommand{\P}{{\mathbb P}}

\newcommand{\rti}{{I_k^{\text{RT}}}}

\newcommand{\hrti}{{\hat{I}_k^{\text{RT}}}}

\newcommand{\bdmi}{{I_k^{\text{BDM}}}}
\newcommand{\bdmiorder}[1]{{I_{#1}^{\text{BDM}}}}
\newcommand{\tbdmi}{{\tilde{I}_k^{\text{BDM}}}}
\newcommand{\hbdmi}{{\hat{I}_k^{\text{BDM}}}}
\newcommand{\hbdmiorder}[1]{{\hat{I}_{#1}^{\text{BDM}}}}
\newcommand{\bbdmi}{{\bar{I}_k^{\text{BDM}}}}

\newcommand{\vertiii}[1]{{\left\vert\kern-0.25ex\left\vert\kern-0.25ex\left\vert #1 
    \right\vert\kern-0.25ex\right\vert\kern-0.25ex\right\vert}}

\newcommand{\MAC}{{\mathit{MAC}}}
\newcommand{\RVP}{{\mathit{RVP}}}

\newcommand{\abssec}[1]{\noindent\small {\bfseries #1\quad}\ignorespaces}
\renewenvironment{abstract}{\abssec{Abstract}}{\par\vspace{1em}}
\newenvironment{keywords}{\abssec{Keywords}}{\par\vspace{1em}}
\newenvironment{MSC}{\abssec{Mathematics Subject Classification (2010)}}{\par\vspace{1em}}

\title{Brezzi-Douglas-Marini interpolation of any order on anisotropic triangles and tetrahedra}
\date{\today}
\author{Thomas Apel \and Volker Kempf}

\begin{document}
\maketitle

\begin{abstract}
	Recently, the $\vec{H}(\operatorname{div})$-conforming finite element families for second order elliptic problems have come more into focus, since due to hybridization and subsequent advances in computational efficiency their use is no longer mainly theoretical. Their property of yielding exactly divergence-free solutions for mixed problems makes them interesting for a variety of applications, including incompressible fluids. In this area, boundary and interior layers are present, which demand the use of anisotropic elements. 
	
	While for the Raviart-Thomas interpolation of any order on anisotropic tetrahedra optimal error estimates are known, this contribution extends these results to the Brezzi-Douglas-Marini finite elements. Optimal interpolation error estimates are proved under two different regularity conditions on the elements, which both relax the standard minimal angle condition. Additionally a numerical application on the Stokes equations is presented to illustrate the findings.
\end{abstract}
\begin{keywords}
	anisotropic finite elements, interpolation error estimate, Brezzi-Douglas-Marini element, maximal angle condition, regular vertex property
\end{keywords}
\begin{MSC}
	65D05 $\cdot$ 65N30  
\end{MSC}

\section{Introduction}
The Brezzi-Douglas-Marini finite element was introduced for two dimensions in \cite{BrezziDouglasMarini1985} and generalized to the three-dimensional case in \cite{Nedelec1986}. 
Similarly to the Raviart-Thomas element, see \cite{RaviartThomas1977,Nedelec1980}, it is commonly used for $\vec{H}(\operatorname{div})$-conforming approximation of second order elliptic problems \cite{BrezziDouglasDuranFortin1987}. 
Applications include mixed methods for incompressible flow problems, as seen in e.g. \cite{SchroederLehrenfeldLinkeLube2018,SchroederLube2018,CockburnKanschatSchotzauSchwab2002,CockburnKanschatSchotzau2007}, where the elements by construction yield divergence-free approximations of the solution. 
In the mentioned references, the used meshes are restricted to shape regular, i.e. isotropic, triangulations, where for the lowest order case the interpolation error estimate, see \cite{BrezziDouglasMarini1985},
\begin{equation} \label{eq:intro_interpolation_estimate}
	\norm{\vec{v} - \bdmi \vec{v}}_{0,T} \lesssim h_T \abs{\vec{v}}_{1,T},
\end{equation}
holds. A possible proof follows by a Bramble-Hilbert type lemma, see e.g. \cite{DupontScott1980}, on a reference element $\hat{T}$, and a subsequent transformation $\vec{x} = J_T \hat{\vec{x}} + \vec{x}_0$ to the element $T$ via the contra-variant Piola transformation
\begin{equation*}
	\vec{v} = \frac{1}{\det J_T} J_T \hat{\vec{v}}.
\end{equation*}

For an anisotropic element $T$, where we have e.g. $h_3 \gg h_1, h_2$, this approach can lead to an estimate of the type 
\begin{align*}
	\norm{\vec{v} - \bdmi \vec{v}}_{0,T} \lesssim& \sum_{\abs{\alpha}\leq 1} h^\alpha \norm{D^\alpha v_1}_{0,T} \left(1+\frac{h_2}{h_1}+ \frac{h_3}{h_1}\right) \\
	&+ \sum_{\abs{\alpha}\leq 1} h^\alpha \norm{D^\alpha v_2}_{0,T} \left(1+\frac{h_1}{h_2}+ \frac{h_3}{h_2}\right) \\
	&+ \sum_{\abs{\alpha}\leq 1} h^\alpha \norm{D^\alpha v_3}_{0,T} \left(1+\frac{h_1}{h_3}+ \frac{h_2}{h_3}\right).
\end{align*}
The potentially huge terms $\frac{h_3}{h_1}, \frac{h_3}{h_2}$ on the right hand side would not appear, if stability estimates
\begin{equation*}
	\norm{(\bdmi \vec{v})_i}_{0,T} \lesssim \norm{v_i}_{1,T}
\end{equation*}
could be proved. Unfortunately this is not a valid estimate, since it would imply, as observed in \cite{AcostaApelDuranLombardi2011} for the Raviart-Thomas element, that 
\begin{equation*}
	v_i \equiv 0 \quad \Rightarrow \quad (\bdmi \vec{v})_i \equiv 0,
\end{equation*}
which is not valid in general, e.g. for $\vec{v} = (0,0, x_1^2)^T$, see also \autoref{subsec:necessity_MAC}. So new stability estimates are needed to incorporate anisotropic elements in the theory.

We consider two conditions on elements, which relax the usual minimum angle condition to allow for anisotropic elements. Both generalize the classical maximum angle condition for triangles from \cite{Synge1957} to three dimensions. The first one was introduced in \cite{Krizek1992} and is widely used, see e.g. \cite{AcostaApelDuranLombardi2011,AcostaDuran1999,Apel1999,DuranLombardi2008} and is a quite literal generalization: An element satisfies a \emph{maximum angle condition}, if all angles are uniformly bounded away from $\pi$. The second condition from \cite{AcostaDuran1999} is more technical, and in three dimensions more restrictive, yielding a subset of elements satisfying a maximum angle condition: An element satisfies a \emph{regular vertex property}, if there is a vertex, for which the outgoing vectors along the edges are uniformly linearly independent. Proper definitions will be given in \autoref{sec:notation_definitions}.

For the Raviart-Thomas interpolation, optimal results for anisotropic interpolation are known, see e.g. \cite{DuranLombardi2008,AcostaApelDuranLombardi2011,AcostaDuran1999}. Starting from the stability estimate
\begin{equation*}
	\norm{(\hrti \hat{\vec{v}})_i}_{0,\hat{T}} \lesssim \norm{\hat{v}_i}_{1,\hat{T}} + \norm{\divhat \hat{\vec{v}}}_{0,\hat{T}},\quad i=1,2,3,
\end{equation*}
see \cite[Lemma 3.3]{AcostaApelDuranLombardi2011}, on the reference element $\hat{T}$, see Figures \ref{fig:ReferenceTriangle} and \ref{fig:ReferenceTetrahedron}, the authors get in \cite[Theorem 3.1]{AcostaApelDuranLombardi2011} to the stability estimate
\begin{equation*}
	\norm{\rti \vec{v}}_{0,{T}} \lesssim \norm{\vec{v}}_{0,{T}} + \sum_{j=1}^3 h_j \norm{\pdv{\vec{v}}{\vec{l}_j}}_{0,T} + \norm{\divt \vec{v}}_{0,{T}},
\end{equation*}
for a general element satisfying a regular vertex property, where $\pdv{}{\vec{l}_i}$ is the directional derivative in the directions $\vec{l}_i$, which in our case are certain unit vectors along edges of the element. Using a Bramble-Hilbert type argument, the interpolation error estimates on these elements are shown in \cite[Theorem 6.2]{AcostaApelDuranLombardi2011} and for $k \geq 0$, $0\leq m \leq k$ take the form
\begin{equation}\label{eq:introduction_interpolation_error}
	\norm{\vec{v}-\rti\vec{v}}_{0,T} \lesssim \sum_{\abs{\alpha} = m + 1} h^\alpha \norm{D^\alpha_{\vec{l}}\vec{v}}_{0,T} + h_T^{m+1} \norm{D^m\divt \vec{v}}_{0,T},
\end{equation}
where we use the a multi-index notation for the directional derivatives $D^\alpha_{\vec{l}} = \frac{\partial^{\abs{\alpha}}}{\partial \vec{l}_1^{\alpha_1}\ldots \partial \vec{l}_d^{\alpha_d}}$.
This type of estimate is particularly useful when $\divt \vec{v} = 0$, so that the estimate reduces to a purely anisotropic estimate in the spirit of \cite{Apel1999}. 

For elements satisfying a maximum angle condition, but no regular vertex condition, a weaker stability estimate is obtained on the reference element $\bar{T}$, see \autoref{fig:ReferenceTetrahedron_noRVP}, in \cite[Lemma 4.3]{AcostaApelDuranLombardi2011}, which leads by similar arguments to an interpolation error estimate like \eqref{eq:intro_interpolation_estimate}, see \cite[Theorem 6.3]{AcostaApelDuranLombardi2011}, but with a relaxed condition on the triangulation. 
The discussion from \cite[Section 5]{AcostaApelDuranLombardi2011} shows that this estimate is sharp.
Note that the maximum angle condition is a necessary condition, see \autoref{subsec:necessity_MAC}.

Following the approach from the reference, we prove the analogs of \cite[Lemma 3.3 and Lemma 4.3]{AcostaApelDuranLombardi2011} for the Brezzi-Douglas-Marini interpolation. We then continue proving the stability estimates on general elements satisfying a regular vertex property and a maximum angle property without regular vertex property, and conclude the interpolation error estimate in the same way as in \cite{AcostaApelDuranLombardi2011}. While in this contribution we only consider the Hilbert space case, all estimates can be proved in the $L^p$ norms, $1\leq p \leq \infty$, as shown in \cite{AcostaApelDuranLombardi2011}. With the same arguments one can prove an analogous estimate to \eqref{eq:introduction_interpolation_error} for prismatic elements and Brezzi-Douglas-Marini interpolation.

In \cite{FarhloulNicaisePaquet2001}, for the lowest order Raviart-Thomas element, the stronger interpolation error estimates on anisotropic triangulations of prismatic domains were shown, by doing an intermediate step of first interpolating to the Raviart-Thomas function space on anisotropic prisms and then interpolating to the simplicial partition of these prisms. This approach is necessary, as one of the tetrahedra from the partition of each prism does not satisfy the regular vertex property, as shown in \autoref{fig:Pentahedron}, so direct interpolation does not yield the desired estimate. A similar approach for the lowest order Brezzi-Douglas-Marini element unfortunately does not work, as the function space on the prismatic triangulation now contains bilinear terms, and as shown in \autoref{rem:weaker_estimate}, interpolation of this type of function on a tetrahedron without a regular vertex property does not satisfy the stronger estimate of the type \eqref{eq:introduction_interpolation_error} for the Brezzi-Douglas-Marini interpolation.

The outline of the article is the following: In \autoref{sec:notation_definitions} we introduce notation, definitions and known results.
In \autoref{sec:stability_estimates} we prove the stability estimates for elements satisfying the maximum angle condition or additionally the regular vertex property.
The proof of the stability of the Raviart-Thomas interpolant in \cite{AcostaApelDuranLombardi2011} is mainly extendable to the Brezzi-Douglas-Marini element, so we follow the reference rather closely.
These results are then employed to prove the final interpolation error estimates in \autoref{sec:interpolation_estimates}.
In the last section, we show, by use of a numerical example taken from \cite{CreuseNicaiseKunert2004}, the application of the Brezzi-Douglas-Marini element to the Stokes equations, using the discontinuous Galerkin discretization from \cite{SchroederLehrenfeldLinkeLube2018}, with the goal of examining the effect of using triangulations with large aspect ratios.

\begin{figure}[htp]
	\centering
	\includegraphics{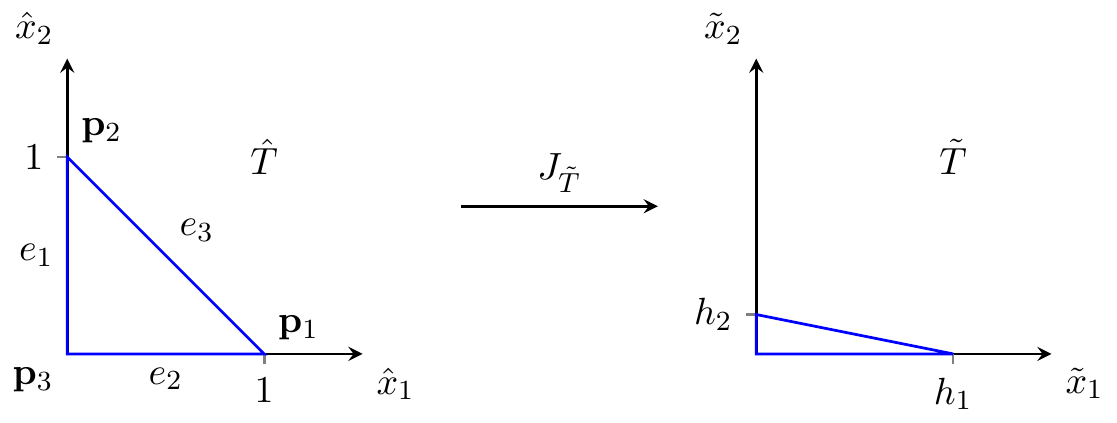}
	\caption{Reference triangle $\hat{T}$ with vertex and edge numbering, transformed triangle}\label{fig:ReferenceTriangle}
	\includegraphics{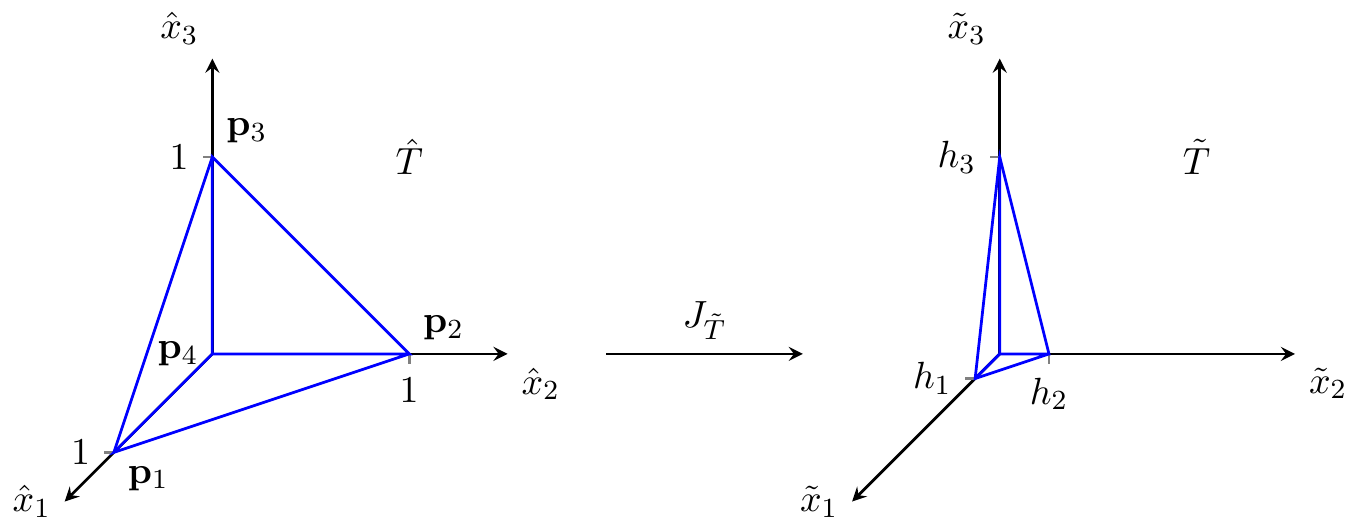}
	\caption{Reference tetrahedron $\hat{T}$ with vertex numbering, transformed tetrahedron of reference family $\mathcal{T}_1$}\label{fig:ReferenceTetrahedron}
	\includegraphics{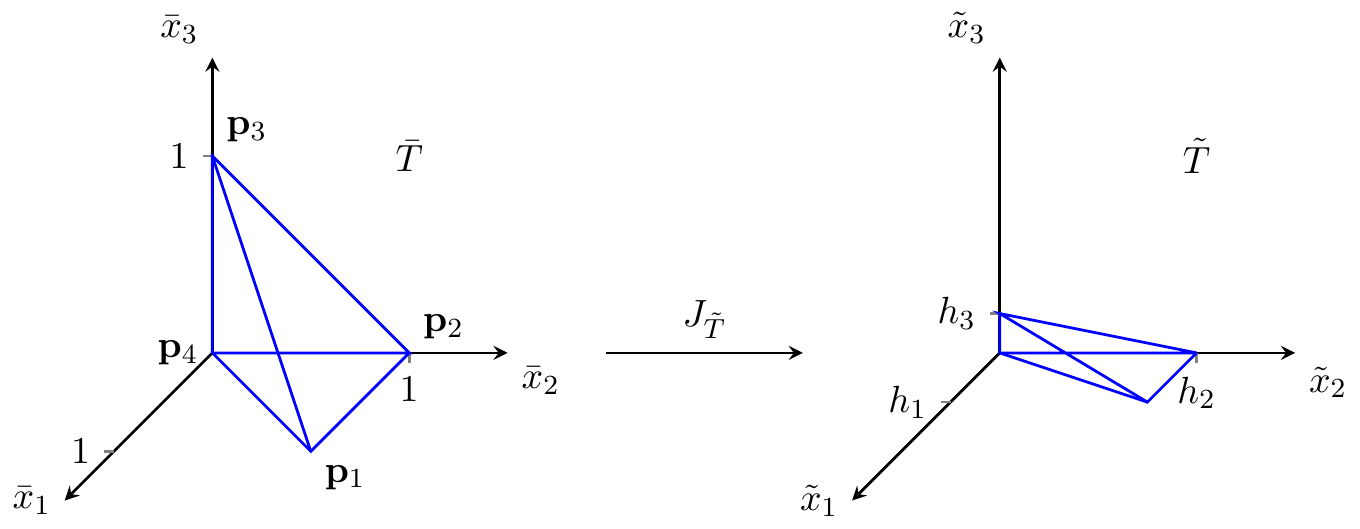}
	\caption{Reference tetrahedron $\bar{T}$ for family of tetrahedra without regular vertex property, transformed tetrahedron of reference family $\mathcal{T}_2$}\label{fig:ReferenceTetrahedron_noRVP}
\end{figure}

\begin{figure}[t]
	\centering
	\includegraphics{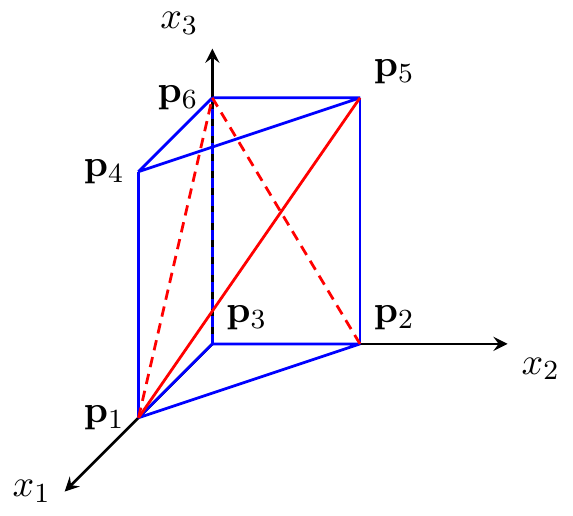}
	\caption{Subdivision of triangular prism $\vec{p_1}\vec{p_2}\vec{p_3}\vec{p_4}\vec{p_5}\vec{p_6}$ in three tetrahedra $\vec{p_1}\vec{p_4}\vec{p_5}\vec{p_6}$, $\vec{p_1}\vec{p_2}\vec{p_3}\vec{p_6}$ and $\vec{p_1}\vec{p_2}\vec{p_5}\vec{p_6}$}\label{fig:Pentahedron}
\end{figure}

\section{Preliminaries}\label{sec:notation_definitions}
\subsection{Notation}
In this text we use the symbol $\lesssim$, meaning less than or equal to up to a positive multiplicative constant. By $D^\alpha$ we denote the derivative due to the multi-index $\alpha$, see e.g. \cite{AdamsFournier2003}, and we use the shorthand $h^\alpha = \prod_{i=1}^d h_i^{\alpha_i}$, where $d\in\{2,3\}$ is the space dimension. The index sets $I_n = \{1,\ldots,n\}$,  $n\in \N$, and for $i\in I_n$ the reduced set ${}_i I_n = I_n \setminus \{i\}$ are used frequently. We choose this notation of the index sets due to their frequent occurrence and in order to keep the content compact.

Vectors, vector valued functions and the spaces of such functions are written in bold letters. The Cartesian unit vectors are denoted by $\vec{e}_i$, $i\in I_d$. 

The symbols $T, \tilde{T}, \hat{T}, \bar{T} \subset \R^d$ denote triangles or tetrahedra, where the hat and bar symbols indicate reference elements, the tilde an element from one of the reference families, see \autoref{subsec:regularity_conditions_elements}. 
Vertices are identified with their position vector $\vec{p}_i$, for the numbering scheme in the reference elements we refer to Figures \ref{fig:ReferenceTriangle}, \ref{fig:ReferenceTetrahedron} and \ref{fig:ReferenceTetrahedron_noRVP}. 
The facets, i.e. edges if $d=2$ and faces if $d=3$, are denoted by $e_i = \operatorname{conv}\{\vec{p}_j:j\in {}_i I_{d+1}\}$, $i\in I_{d+1}$, so that facet $e_i$ is opposite vertex $\vec{p}_i$. The outward facing normal vector on the facet $e_i$ is denoted by $\vec{n}_i$, while for the generic normal vector on $\partial T$ we use $\vec{n}$.

We also employ the standard notation $\norm{\cdot}_{k,D}$ and $\abs{\cdot }_{k,D}$ for the norms and semi-norms of order $k$ of the Sobolev spaces $H^k(D)$ on a domain $D\subset \R^n$.

\subsection{Regularity conditions on elements}\label{subsec:regularity_conditions_elements}
We now give precise definitions of the two already mentioned conditions on the elements, that are essential to the analysis.
\begin{definition}\label{def:maximum_angle_condition}
	An element $T$ satisfies the \emph{maximum angle condition} with a constant $\bar{\phi} < \pi$, written as $\MAC(\bar{\phi})$, if the maximum angle between facets and, for $d=3$, the maximum angle inside the facets are less than or equal to $\bar{\phi}$.
\end{definition}
This condition is a generalization due to \cite{Krizek1992} of the maximum angle condition in triangles, which was first used in \cite{Synge1957}, and is very common when dealing with anisotropic elements, see e.g. \cite{AcostaApelDuranLombardi2011,AcostaDuran1999,Apel1999,DuranLombardi2008}. The next property is equivalent to the maximum angle condition for $d=2$, see \cite{AcostaApelDuranLombardi2011}, while in three dimensions it describes a proper subclass. 
\begin{definition}\label{def:regular_vertex_property}
	An element $T$ satisfies the \emph{regular vertex property} with a constant $\bar{c}$, written as $\RVP(\bar{c})$, if there is a vertex $\vec{p}_k$ of $T$, so that for the matrix $N_k$, made up of the unit column vectors $\vec{l}^k_j=\frac{\vec{p}_j-\vec{p}_k}{\norm{\vec{p}_j-\vec{p}_k}}$ outgoing from vertex $\vec{p}_k$ towards vertex $\vec{p}_j$, $j\in {}_k I_{d+1}$, the inequality 
	\begin{equation*}
		|\det N_k| \geq \bar{c} > 0
	\end{equation*}
	holds. The vertex $\vec{p}_k$ is then called \emph{regular vertex} of the element $T$. Without loss of generality for the rest of the text we assume that the vertices are numbered so that $\vec{p}_{d+1}$ is the regular vertex, so that we can use the more intuitive notation $\vec{l}_i = \vec{l}^{d+1}_i$, $i\in I_d$ and the element size parameters $h_i$, $i\in I_d$, which are defined as the lengths of the edges corresponding to the vectors $\vec{l}_i$. 
\end{definition}

In the text we consider mainly two reference elements, $\hat{T}$ and $\bar{T}$, see the left sides of Figures \ref{fig:ReferenceTriangle}, \ref{fig:ReferenceTetrahedron} and \ref{fig:ReferenceTetrahedron_noRVP}. In addition, we introduce the two reference families $\mathcal{T}_1$ and $\mathcal{T}_2$ of elements with vertices at $0, h_1\vec{e}_1, h_2\vec{e}_2, h_3\vec{e}_3$ and respectively $0, h_1\vec{e}_1+h_2\vec{e}_2,h_2\vec{e}_2, h_3\vec{e}_3$, with arbitrary size parameters $h_i$, $i\in I_d$, which can be seen on the right sides of these figures. 

The next two lemmas are taken from \cite[Theorem 2.2, Theorem 2.3]{AcostaApelDuranLombardi2011} and are stated without proof. They show that these reference families are sufficient to get any tetrahedron satisfying $\MAC(\bar{\phi})$, resp. $\RVP(\bar{c})$, by a reasonable affine transformation $F$, i.e. $F(\tilde{\vec{x}}) = J_T  \tilde{\vec{x}} + \vec{x}_0$, $J_T\in \R^{d\times d}$, where $\norm{J_T}_{\infty}, \norm{J_T^{-1}}_{\infty} \leq C$, with a constant $C$ which only depends on $\bar{\phi}$ resp. $\bar{c}$.
\begin{lemma}\label{lem:MAC}
	Let $T$ be an element satisfying $\MAC(\bar{\phi})$. Then there is an element $\tilde{T}\in \mathcal{T}_1 \cup \mathcal{T}_2$, so that an affine transformation $F(\tilde{\vec{x}}) = J_T \tilde{\vec{x}} + \vec{x}_0$, with $\norm{J_T}_{\infty}, \norm{J_T^{-1}}_{\infty} \leq C$, exists that maps $\tilde{T}$ onto $T$, where $C$ only depends on $\bar{\phi}$.
\end{lemma}

\begin{lemma}\label{lem:RVP}
	Let $T$ be an element satisfying $\RVP(\bar{c})$. Then there is an element $\tilde{T}\in \mathcal{T}_1$, so that an affine transformation $F(\tilde{\vec{x}}) = J_T \tilde{\vec{x}} + \vec{x}_0$, with $\norm{J_T}_{\infty}, \norm{J_T^{-1}}_{\infty} \leq C$, exists that maps $\tilde{T}$ onto $T$, where $C$ only depends on $\bar{c}$. Additionally, if the edges of $T$ sharing the regular vertex $\vec{p}_{d+1}$ have lengths $h_i$, $i\in I_d$, then we can take a $\tilde{T} \in \mathcal{T}_1$ with lengths $h_i$ in direction $\tilde{x}_i$.
\end{lemma}

A drawback of dealing with tetrahedra satisfying the regular vertex property is, that they can not be used to fill arbitrary volumes. For example, consider an anisotropic pentahedron, i.e. a triangular prism. This volume can be subdivided into three tetrahedra, see \autoref{fig:Pentahedron}, of which only two satisfy the regular vertex property, while the third ($\vec{p}_1\vec{p}_2\vec{p}_5\vec{p}_6$ in \autoref{fig:Pentahedron}) does not and is of the type pictured in \autoref{fig:ReferenceTetrahedron_noRVP}.

\subsection{Interpolation operator}
The Brezzi-Douglas-Marini finite element was introduced for triangles in \cite{BrezziDouglasMarini1985} and generalized by N\'{e}d\'{e}lec in \cite{Nedelec1986}. On an arbitrary element $T$, the Brezzi-Douglas-Marini function space of order $k$ is $\vec{P}_k(T)=(P_k(T))^d$, the full space of polynomials of order $k$. For the definition of the interpolation operator, we need to introduce the spaces, see \cite{Nedelec1986,BrezziDouglasMarini1985},
\begin{align*}
	\vec{Q}_k(T) &= \{\vec{z}\in \vec{P}_k(T): \nabla\cdot \vec{z} = 0, \vec{z} \cdot \vec{n}|_{\partial T} =0\}, \\
	\vec{N}_k(T) &= \vec{P}_{k-1}(T) \oplus \vec{S}_{k}(T), \\
	\vec{S}_k(T) &= \{\vec{p}\in \vec{P}_k(T): \vec{p}\cdot\vec{x} \equiv 0\}.
\end{align*}
In the literature, there are two ways to define the interpolation operator $\bdmi$ of the Brezzi-Douglas-Marini element of order $k\geq 1$. The first one, from \cite{BrezziDouglasMarini1985,BrezziFortin1991} defines the operator by the relations
\begin{align}
	\int_{e_i} (\bdmi {\vec{v}})\cdot \vec{n}_i z &= \int_{e_i} {\vec{v}}\cdot\vec{n}_i z, && \forall z \in P_k(e_i), \quad i\in I_{d+1}, \label{eq:BDM_dofs_old_1}\\
	\int_T (\bdmi \vec{v}) \cdot \nabla z &= \int_T \vec{v} \cdot \nabla z, && \forall z \in P_{k-1}(T), \label{eq:BDM_dofs_old_2} \\
	\int_T (\bdmi \vec{v}) \cdot \vec{z} &= \int_T \vec{v} \cdot \vec{z}, && \forall \vec{z} \in \vec{Q}_k(T), \label{eq:BDM_dofs_old_3}
\end{align}
while the second, see \cite{Nedelec1986,BoffiBrezziFortin2013}, uses the relations
\begin{align}
	\int_{e_i} (\bdmi {\vec{v}})\cdot \vec{n}_i z &= \int_{e_i} {\vec{v}}\cdot\vec{n}_i z, && \forall z \in P_k(e_i), \quad i\in I_{d+1}, \label{eq:BDM_dofs_1}\\
	\int_T (\bdmi {\vec{v}})\cdot \vec{z} &= \int_T \vec{v} \cdot \vec{z}, && \forall \vec{z} \in \vec{N}_{k-1}(T), \text{ if } k \geq 2. \label{eq:BDM_dofs_2}
\end{align}
For the purpose of this paper we use the latter definition, and we show in \autoref{rem:diff_dofs}, that there is a significant difference in using one or the other definition for the proof of our stability estimates.

We indicate the element, the operator interpolates to by the hat, bar and tilde symbol, e.g. $\hbdmi$ corresponds to $\hat{T}$.

\section{Stability estimates}\label{sec:stability_estimates}
\subsection{Stability with regular vertex property} \label{subsec:refMAC+RVP}
In order to get a stability estimate on the reference element $\hat{T}$, for $d\in\{2,3\}$, we need a technical lemma first, which is the analog to \cite[Lemma 3.2]{AcostaApelDuranLombardi2011}. We formulate the lemma for three dimensions, but the statement holds for $d=2$ as well.
\begin{lemma}\label{lem:bdm_zero}
	On the reference element $\hat{T}$ let $\hat{f}_i \in {L}^2(e_i)$, $i\in I_d$, and
	\begin{align*}
		\hvec{u}(\hvec{x}) = (\hat{f}_1(\hat{x}_2,\hat{x}_3), 0, 0)^T, \quad \hvec{v}(\hvec{x}) = (0,\hat{f}_2(\hat{x}_1,\hat{x}_3),0)^T, \quad \hvec{w}(\hvec{x}) = (0,0,\hat{f}_3(\hat{x}_1,\hat{x}_2))^T.
	\end{align*}
	Then there are functions $\hat{q}_i\in P_k(e_i)$, $i\in I_d$, so that
	\begin{align*}
		\hbdmi\hvec{u} = (\hat{q}_1(\hat{x}_2,\hat{x}_3), 0, 0)^T, \; \hbdmi\hvec{v} = (0,\hat{q}_2(\hat{x}_1,\hat{x}_3),0)^T, \; \hbdmi\hvec{w} = (0,0,\hat{q}_3(\hat{x}_1,\hat{x}_2))^T.
	\end{align*}
\end{lemma}
\begin{proof}
	We prove the statement by showing that the functions $\hat{q}_i$ are uniquely defined by the relations of the interpolation operator from \eqref{eq:BDM_dofs_1}, \eqref{eq:BDM_dofs_2}. We show this in detail for the first case, the results for $\hbdmi \hvec{v}$ and $\hbdmi \hvec{w}$ follow analogously.
	
	The normal vectors $\vec{n}_j$ of the facets $e_j$, $j\in I_{d+1}$ of the reference element $\hat{T}$, see \autoref{fig:ReferenceTetrahedron}, are $\vec{n}_1 = (-1,0,0)^T$, $\vec{n}_2 = (0,-1,0)^T$, $\vec{n}_3 = (0,0,-1)^T$ and $\vec{n}_4 = \frac{1}{\sqrt{3}}(1,1,1)^T$. So the first relation from \eqref{eq:BDM_dofs_1} reduces to 
	\begin{equation}\label{eq:bdm_zero}
		\int_{e_1} \hat{f}_1 z = \int_{e_1} \hat{q}_1 z, \quad \forall z \in P_k(e_1),
	\end{equation}
	which already defines $\hat{q}_1$ uniquely. In the rest of the proof, we show that the remaining relations are compatible.
	For $j=2,3$ we get trivial equalities from \eqref{eq:BDM_dofs_1}. For $j=4$ we get 
	\begin{align*}
		\int_{e_4} (\hvec{u} - \hbdmi \hvec{u})\cdot \vec{n}_4 z &= \int_0^1 \int_0^{1-\hat{x}_2} (\hat{f}_1(\hat{x}_2,\hat{x}_3) - \hat{q}_1(\hat{x}_2,\hat{x}_3))z(1-\hat{x}_2-\hat{x}_3,\hat{x}_2,\hat{x}_3) \dd \hat{x}_3 \dd \hat{x}_2 \\
		&= \int_{e_1} (\hat{f}_1(\hat{x}_2,\hat{x}_3) - \hat{q}_1(\hat{x}_2,\hat{x}_3))z(1-\hat{x}_2-\hat{x}_3,\hat{x}_2,\hat{x}_3) \dd \hat{x}_3 \dd \hat{x}_2 = 0
	\end{align*}
	where the last equality holds due to \eqref{eq:bdm_zero} since $z(1-\hat{x}_2-\hat{x}_3,\hat{x}_2,\hat{x}_3)$ is a polynomial of degree $k$ in the variables $\hat{x}_2,\hat{x}_3$. 
	For the internal degrees of freedom, take an arbitrary $\vec{z} \in \vec{N}_{k-1}(\hat{T})\subset \vec{P}_{k-1}(\hat{T})$ and let $Z\in P_k(\hat{T})$ be so that $\pdv{Z}{\hat{x}_1} = z_1$. Then with $\vec{n}= (n_1,n_2,n_3)^T$ being the outward normal vector on $\partial \hat{T}$, using integration by parts and noting that $\hat{f}_1$ and $\hat{q}_1$ do not depend on $\hat{x}_1$, we can calculate
	\begin{align*}
		\int_{\hat{T}} \hvec{u} \cdot \vec{z} &= \int_{\hat{T}} \hat{f}_1 z_1 = \int_{\hat{T}} \hat{f}_1 \pdv{Z}{\hat{x}_1} = \int_{\partial \hat{T}} \hat{f}_1 Z n_1 - \int_{\hat{T}} \pdv{\hat{f}_1}{\hat{x}_1}Z \\
		&= \int_{\partial\hat{T}} \hat{q}_1 Z n_1 = \int_{\hat{T}} \hat{q}_1 \pdv{Z}{\hat{x}_1} = \int_{\hat{T}} \hat{q}_1 z_1,
	\end{align*}
	which concludes the proof.
\end{proof}

\begin{remark}\label{rem:diff_dofs}
The two different definitions of the Brezzi-Douglas-Marini degrees of freedom, see \eqref{eq:BDM_dofs_old_1}--\eqref{eq:BDM_dofs_old_3} and \eqref{eq:BDM_dofs_1}--\eqref{eq:BDM_dofs_2}, both describe polynomial approximations of order $k\in\N$ of $\vec{H}(\operatorname{div})$, but the associated interpolation operators differ significantly. As mentioned before, we use the definitions from \cite{Nedelec1986} in order to prove the Lemmas \ref{lem:bdm_zero} and \ref{lem:bdm_zero_noRVP}, which are then used to get the subsequent stability estimates. The statements of these Lemmas do not hold, were the interpolation operator defined by the degrees of freedom from \cite{BrezziDouglasMarini1985}, as we will show in a straightforward example. 

The difference only appears for $k\geq 2$, since for $k=1$ the degrees of freedom are the same. So consider the case $k=2$, and the function $\hvec{v} = (0,\hat{x}_1^3)^T$ on the reference triangle $\hat{T}$, see \autoref{fig:ReferenceTriangle}. Then using the interpolation operator defined by \eqref{eq:BDM_dofs_1}--\eqref{eq:BDM_dofs_2}, the interpolated function is 
\begin{equation*}
	\hbdmiorder{2}\hvec{v} = \begin{pmatrix} 0 \\ \frac{1}{20} - \frac{3}{5} \hat{x}_1 + \frac{3}{2} \hat{x}_1^2 \end{pmatrix},
\end{equation*}
which as expected has the properties described in \autoref{lem:bdm_zero}.
Calculating the interpolant with the operator defined by \eqref{eq:BDM_dofs_old_1}--\eqref{eq:BDM_dofs_old_3}, we get the function
\begin{equation*}
	\hat{I}^{\overline{BDM}}_2\hvec{v} = \begin{pmatrix} \frac{3}{140}\hat{x}_1(1-\hat{x}_1-2\hat{x}_2) \\ \frac{1}{20} - \frac{3}{5} \hat{x}_1 + \frac{3}{2} \hat{x}_1^2 - \frac{3}{140} \hat{x}_2(1-2\hat{x}_1-\hat{x}_2) \end{pmatrix},
\end{equation*}
which clearly does not have the desired property.
\end{remark} \pagebreak[3]

We can now show the stability estimate on the reference element $\hat{T}$.
\begin{lemma}\label{lem:Estimate_reference_element}
	Let $\hvec{u} \in \vec{H}^1(\hat{T})$, then the estimates
	\begin{equation} \label{eq:estimate_ref_elem}
		\norm{(\hbdmi\hvec{u})_i}_{0,\hat{T}} \lesssim \norm{{\hat{u}_i}}_{1,\hat{T}} + \norm{\divhat \hvec{u}}_{0,\hat{T}}, \quad i\in I_d,
	\end{equation}
	hold.
\end{lemma}
\begin{proof}
	We follow the lines of the proof of \cite[Lemma 3.3]{AcostaApelDuranLombardi2011}. Again we detail the proof for the case $i=1$, the other estimates follow analogously. 
	
	Let $\hvec{u}_* = (0, \hat{u}_2(\hat{x}_1,0,\hat{x}_3), \hat{u}_3(\hat{x}_1,\hat{x}_2,0)^T$ and set
	\begin{equation}
		\hvec{v} = \hvec{u} - \hvec{u}_* =  	\begin{pmatrix}
												\hat{u}_1 \\
												\hat{u}_2 - \hat{u}_2(\hat{x}_1,0,\hat{x}_3) \\
												\hat{u}_3 -  \hat{u}_3(\hat{x}_1,\hat{x}_2,0)
											\end{pmatrix}.
	\end{equation}
	Then we know from \autoref{lem:bdm_zero}, that $(\hbdmi \hvec{v})_1 = (\hbdmi \hvec{u})_1$ and additionally $\divhat \hvec{v} = \divhat \hvec{u} - \divhat \hvec{u}_* = \divhat \hvec{u}$.
	
	Now using $\hvec{v}_* = (0,\hat{x}_2 q_2, \hat{x}_3 q_3)^T$ construct another function 
	\begin{equation}
		\hvec{w} = \hvec{v} - \hvec{v}_* = \begin{pmatrix} \hat{v}_1 \\ \hat{v}_2 - \hat{x}_2 {q}_2 \\ \hat{v}_3 -\hat{x}_3 {q}_3 \end{pmatrix},
	\end{equation}
	where  ${q}_2,{q}_3 \in P_{k-1}(\hat{T})$ are chosen so that
	\begin{align}
		\int_{\hat{T}} \hat{w}_2 z = \int_{\hat{T}} (\hat{v}_2 - \hat{x}_2 q_2) z = 0, \quad \forall z\in P_{k-1}(\hat{T}), \label{eq:q_2}\\
		\int_{\hat{T}} \hat{w}_3 z = \int_{\hat{T}} (\hat{v}_3 - \hat{x}_3 q_3) z = 0, \quad \forall z\in P_{k-1}(\hat{T}). \label{eq:q_3}
	\end{align}
	The functions $q_i$, $i=2,3$, are the projections of $\frac{\hat{v}_i}{\hat{x}_i}$ into $P_{k-1}(\hat{T})$ with respect to the weighted scalar product $(q,z)= \int_{\hat{T}} \hat{x}_i q z$, which are well defined for $\hat{v}_i \in L^2(\hat{T})$. 
	
	Since $\hvec{v}_* \in \vec{P}_k(\hat{T})$, we know by the interpolation property of the operator $\hbdmi$, that $\hbdmi \hvec{v}_* = \hvec{v}_*$, and thus $(\hbdmi \hvec{w})_1 = (\hbdmi \hvec{v})_1 = (\hbdmi \hvec{u})_1$. By \eqref{eq:BDM_dofs_1} and \eqref{eq:BDM_dofs_2}, the interpolated function $\hbdmi \hvec{w} = \hvec{p} = (\hat{p}_1,\hat{p}_2,\hat{p}_3)^T$ is thus defined by the relations
	\begin{align*}
		\int_{e_1} \hat{p}_1 z &= \int_{e_1} \hat{w}_1 z = \int_{e_1} \hat{u}_1 z, && \forall z \in P_k(e_1),\\
		\int_{e_2} \hat{p}_2 z &= \int_{e_2} \hat{w}_2 z = \int_{e_2} (\hat{v}_2 - \hat{x}_2 q_2) z = 0, && \forall z \in P_k(e_2), \\
		\int_{e_3} \hat{p}_3 z &= \int_{e_3} \hat{w}_3 z = \int_{e_3} (\hat{v}_3 - \hat{x}_3 q_3) z = 0, && \forall z \in P_k(e_3), \\
		\int_{e_4} (\hvec{p} \cdot \vec{n}_4) z &= \int_{e_4} (\hvec{w} \cdot \vec{n}_4)z \\
		&= \int_{\hat{T}} (\divhat \hvec{w}) z + \int_{\hat{T}} \hvec{w} \cdot \nabla z - \sum_{i=1}^3 \int_{e_i} (\hvec{w} \cdot \vec{n}_i)z \\
		&= \int_{\hat{T}} (\divhat \hvec{w}) z + \int_{\hat{T}} \hat{w}_1 \pdv{z}{\hat{x}_1} - \int_{e_1} \hat{w}_1 z, && \forall z \in P_k(e_4), \\
		\int_{\hat{T}} \hvec{p}\cdot \vec{z} &= \int_{\hat{T}} \hvec{w} \cdot \vec{z}, &&\forall \vec{z} \in \vec{N}_{k-1}(\hat{T}).
	\end{align*}
	Now recall that $\vec{N}_{k-1}(\hat{T}) = \vec{P}_{k-2}(\hat{T}) \oplus \vec{S}_{k-1}(\hat{T})$, then the last equation is equivalent to the individual relations
	\begin{align*}
		\int_{\hat{T}} \hat{p}_1 z &= \int_{\hat{T}} \hat{w}_1 z, &&\forall z\in P_{k-2}(\hat{T}),\\
		\int_{\hat{T}} \hat{p}_2 z &= \int_{\hat{T}} \hat{w}_2 z = 0, &&\forall z\in P_{k-2}(\hat{T}),\\
		\int_{\hat{T}} \hat{p}_3 z &= \int_{\hat{T}} \hat{w}_3 z = 0, &&\forall z\in P_{k-2}(\hat{T}),\\
		\int_{\hat{T}} \hvec{p} \cdot \vec{z} &= \int_{\hat{T}} \hvec{w} \cdot \vec{z} = \int_{\hat{T}} (\hat{w}_1 z_1 + \hat{w}_2 z_2 + \hat{w}_3 z_3) =  \int_{\hat{T}} \hat{w_1}z_1, && \forall \vec{z} \in \vec{S}_{k-1}(\hat{T}),
	\end{align*}
	where we used in the calculations that $q_2$, $q_3$ are chosen according to \eqref{eq:q_2}, \eqref{eq:q_3}. Hence we know, that the interpolant $\hvec{p}$ is defined by the terms $\int_{e_1} \hat{w}_1 z = \int_{e_1} \hat{u}_1 z$, $\int_{\hat{T}} \hat{w}_1 z = \int_{\hat{T}} \hat{u}_1 z$ and $\int_{\hat{T}} (\divhat \hvec{w}) z = \int_{\hat{T}} (\divhat (\hvec{u} - \hvec{v}_*) z$. Thus we can estimate
	\begin{align}
		\norm{(\hbdmi \hvec{u})_1}_{0,\hat{T}} = \norm{\hvec{p}}_{0,\hat{T}} &\leq \norm{\hat{u}_1}_{0,e_1} + \norm{\hat{u}_1}_{0,\hat{T}} + \norm{\divhat \hvec{u}}_{0,\hat{T}} + \norm{\divhat \hvec{v}_*}_{0,\hat{T}} \nonumber\\
		&\leq \norm{\hat{u}_1}_{1,\hat{T}} + \norm{\divhat \hvec{u}}_{0,\hat{T}} + \norm{\divhat \hvec{v}_*}_{0,\hat{T}}, \label{eq:bdmi_u_1}
	\end{align}
	where we used a trace theorem. In order to get to the desired result, we need to estimate $\norm{\divhat \hvec{v}_*}_{0,\hat{T}}$, where we proceed similarly to \cite{AcostaApelDuranLombardi2011}. 
	
	Set $\hvec{v}_0 = (0,\hat{v}_2,\hat{v}_3)^T$. For all $z\in P_k(\hat{T})$ it holds that $\nabla z \in \vec{P}_{k-1}(\hat{T})$, and thus, using the definitions of $q_2$, $q_3$ from \eqref{eq:q_2}, \eqref{eq:q_3} we get
	\begin{align}
		0 &= \int_{\hat{T}} (\hvec{v}_0 - \hvec{v}_*) \cdot \nabla z = \int_{\partial \hat{T}} (\hvec{v}_0 - \hvec{v}_*) \cdot \vec{n} z - \int_{\hat{T}} \divhat(\hvec{v}_0-\hvec{v}_*) z \nonumber\\
		&= \int_{e_4} (\hvec{v}_0 - \hvec{v}_*) \cdot \vec{n}_4 z - \int_{\hat{T}} \divhat (\hvec{v}_0 - \hvec{v}_*)z. \label{eq:bdmi_2}
	\end{align}
	Choose $z = (1-\hat{x}_1-\hat{x}_2-\hat{x}_3) \tilde{z}$, $\tilde{z} \in P_{k-1}(\hat{T})$, then since $z = 0$ on $e_4$, we get from \eqref{eq:bdmi_2}
	\begin{equation}\label{eq:bdmi_3}
		\int_{\hat{T}}(1-\hat{x}_1-\hat{x}_2-\hat{x}_3) (\divhat \hvec{v}_*)\tilde{z} = \int_{\hat{T}}(1-\hat{x}_1-\hat{x}_2-\hat{x}_3) (\divhat \hvec{v}_0)\tilde{z}, \quad \forall \tilde{z} \in P_{k-1}(\hat{T}).
	\end{equation}
	Now choose $\tilde{z} = \divhat \hvec{v}_*$. Then using the equivalence of norms in finite dimensional spaces, \eqref{eq:bdmi_3} and the Cauchy-Schwarz inequality we conclude
	\begin{align*}
		\norm{\divhat \hvec{v}_*}_{0,\hat{T}}^2 &\lesssim \int_{\hat{T}} (1-\hat{x}_1-\hat{x}_2-\hat{x}_3) (\divhat \hvec{v}_*)^2 \\
		&\lesssim \int_{\hat{T}} (1-\hat{x}_1-\hat{x}_2-\hat{x}_3) (\divhat \hvec{v}_*)(\divhat \hvec{v}_0) \\
		&\lesssim \sup_{\hvec{x}\in \hat{T}} (1-\hat{x}_1-\hat{x}_2-\hat{x}_3) \norm{\divhat \hvec{v}_*}_{0,\hat{T}} \norm{\divhat \hvec{v}_0}_{0,\hat{T}}.
	\end{align*}
	Finally we can estimate 
	\begin{align*}
		\norm{\divhat \hvec{v}_*}_{0,\hat{T}} \lesssim \norm{\divhat \hvec{v}_0}_{0,\hat{T}} \leq \norm{\divhat \hvec{v}}_{0,\hat{T}} + \norm{\pdv{\hat{v}_1}{\hat{x}_1}}_{0,\hat{T}} = \norm{\divhat \hvec{u}}_{0,\hat{T}} + \norm{\pdv{\hat{u}_1}{\hat{x}_1}}_{0,\hat{T}},
	\end{align*}
	which combined with \eqref{eq:bdmi_u_1} gets us the final estimate
	\begin{equation*}
		\norm{(\hbdmi \hvec{u})_1}_{0,\hat{T}} \lesssim \norm{\hat{u}_1}_{1,\hat{T}} + \norm{\divhat \hvec{u}}_{0,\hat{T}}. \qedhere
	\end{equation*}
\end{proof}

Consider now the transformation
\begin{equation}\label{eq:generaltransformation}
	\tilde{\vec{x}} = J_{\tilde{T}} \hat{\vec{x}}
\end{equation}
of the reference element $\hat{T}$ on the element $\tilde{T}$ of the reference family $\mathcal{T}_1$, with 
\begin{equation}\label{eq:reference_family_transformation}
	J_{\tilde{T}} = \begin{pmatrix}	h_1 &  & 0 \\ & \ddots & \\ 0 & & h_d\end{pmatrix} \in \R^{d\times d},
\end{equation}
where $h_i$, $i \in I_d$, are the element size parameters pictured in \autoref{fig:ReferenceTriangle} and \ref{fig:ReferenceTetrahedron}.
Then by the contra-variant Piola transformation a function $\hat{\vec{v}} \in \vec{L}^2(\hat{T})$ gets transformed into a function $\tilde{\vec{v}} \in \vec{L}^2(\tilde{T})$, which has the form
\begin{equation*}
	\tilde{\vec{v}}(\tilde{\vec{x}}) = \frac{1}{\det J_{\tilde{T}}} J_{\tilde{T}} \hat{\vec{v}}(\hat{\vec{x}}) = \begin{pmatrix}	{}_1 h^{-1} &  & 0 \\ & \ddots & \\ 0 & & {}_d h^{-1}\end{pmatrix}\hat{\vec{v}}(\hat{\vec{x}}),
\end{equation*}
where we used the shorthand ${}_i h = \prod_{j\in {}_i I_d} h_j$. 

\begin{lemma}\label{lem:transformed_element_MAC+RVP}
	Let $\tilde{\vec{v}} \in \vec{H}^1(\tilde{T})$, where $\tilde{T} = J_{\tilde{T}} \hat{T} + \vec{x}_0$. Then on the transformed element $\tilde{T}$ we have the estimate
	\begin{equation}\label{eq:transformed_element_MAC}
		\norm{\tbdmi \tilde{\vec{v}}}_{0,\tilde{T}} \lesssim \sum_{\abs{\alpha}\leq 1} h^\alpha \norm{D^\alpha \tilde{\vec{v}}}_{0,\tilde{T}} + h_{\tilde{T}} \norm{\divtilde \tilde{\vec{v}}}_{0,\tilde{T}},
	\end{equation}		
	where $h_{\tilde{T}} = \max\{h_i: i\in I_d\}$.
\end{lemma}
\begin{proof}
	The following proof is essentially the proof of \cite[Proposition 3.4]{AcostaApelDuranLombardi2011}.
	By straightforward calculations we observe
	\begin{align}
		\norm{\tilde{\vec{v}}}_{0,\tilde{T}} &= \left(\int_{\tilde{T}}  \sum_{i\in I_d} \tilde{v}_i^2 \right)^{\nicefrac{1}{2}} \nonumber\\
		&\leq  (\det J_{\tilde{T}})^{\nicefrac{1}{2}} \sum_{i\in I_d} {_i h}^{-1}\left(\int_{\hat{T}} \hat{v}_i^2\right)^{\nicefrac{1}{2}} = (\det J_{\tilde{T}})^{\nicefrac{1}{2}} \sum_{i\in I_d} {_i h}^{-1}\norm{\hat{v}_i}_{0,\hat{T}},\label{eq:L2_norm_Piola}
	\end{align}
	and for $i\in I_d$
	\begin{equation}\label{eq:H1_norm_Piola}
		(\det J_{\tilde{T}})^{\nicefrac{1}{2}} \norm{\hat{v}_i}_{1,\hat{T}} = {_i h} \sum_{\abs{\alpha}\leq 1} h^\alpha \norm{D^\alpha \tilde{v}_i}_{0,\tilde{T}}.
	\end{equation}
	Now using \eqref{eq:L2_norm_Piola}, \autoref{lem:Estimate_reference_element} and \eqref{eq:H1_norm_Piola} we get
	\begin{align*}
		\norm{\tbdmi \tilde{\vec{v}}}_{0,\tilde{T}} &\leq (\det J_{\tilde{T}})^{\nicefrac{1}{2}} \sum_{i\in I_d} {_i h}^{-1} \norm{(\hbdmi \hat{\vec{v}})_i}_{0,\hat{T}} \\
		&\lesssim (\det J_{\tilde{T}})^{\nicefrac{1}{2}} \sum_{i\in I_d} {_i h}^{-1} \left( \norm{\hat{v}_i}_{1,\hat{T}} + \norm{\divhat \hat{\vec{v}}}_{0,\hat{T}} \right) \\
		&\lesssim \sum_{i\in I_d} {_i h}^{-1} \left({_i h}\sum_{\abs{\alpha}\leq 1} h^\alpha \norm{D^\alpha \tilde{v}_i}_{0,\tilde{T}} + \det J_{\tilde{T}} \norm{\divtilde \tilde{\vec{v}}}_{0,\tilde{T}}  \right) \\
		&\lesssim \sum_{\abs{\alpha}\leq 1} h^\alpha \norm{D^\alpha \tilde{\vec{v}}}_{0,\tilde{T}} + h_{\tilde{T}} \norm{\divtilde \tilde{\vec{v}}}_{0,\tilde{T}}. \qedhere
	\end{align*}
\end{proof}	

We now get to the main result of this subsection.
\begin{theorem}\label{th:stability_RVP}
	Let the element $T$ satisfy a regular vertex property $\RVP(\bar{c})$, let $\vec{p}_{d+1}$ be the regular vertex, $\vec{l}_i$ and $h_i$, $i\in I_d$, the corresponding vectors and element size parameters from \autoref{def:regular_vertex_property}. Then for $\vec{v} \in \vec{H}^1(T)$ the estimate
	\begin{equation}\label{eq:stability_RVP}
		\norm{\bdmi \vec{v}}_{0,T} \lesssim \norm{\vec{v}}_{0,T} + \sum_{j\in I_d} h_j \norm{\pdv{\vec{v}}{\vec{l}_j}}_{0,T} + h_T\norm{\divt \vec{v}}_{0,T}
	\end{equation}
	holds, where the constant only depends on $\bar{c}$.
\end{theorem}
\begin{proof}
	The steps of the proof are the same as the proof of \cite[Theorem 3.1]{AcostaApelDuranLombardi2011}. For completeness, we repeat it here.
	We assume that the regular vertex $\vec{p}_{d+1}$ is located at the origin. By \autoref{lem:RVP} there exists an element $\tilde{T}\in \mathcal{T}_1$ and a matrix $J_T\in \R^{d\times d}$, so that $\tilde{T}$ is mapped by $\vec{x} = J_T\tilde{\vec{x}}$ onto $T$ and $J_T\vec{e}_i = \vec{l}_i$, $i\in I_d$. Let $\vec{v} \in \vec{H}^1(T)$ be the Piola transform of $\tilde{\vec{v}}\in \vec{H}^1(\tilde{T})$, i.e.
	\begin{equation*}
		\vec{v} (\vec{x}) = (\det J_T)^{-1} J_T\tilde{\vec{v}}(\tilde{\vec{x}}).
	\end{equation*}
	Using \autoref{lem:transformed_element_MAC+RVP} we get
	\begin{equation*}
		\norm{\bdmi \vec{v}}_{0,T}^2 \lesssim \frac{\norm{J_T}_{\infty}^2}{\det J_T} \left(\norm{\tilde{\vec{v}}}_{0,\tilde{T}}^2 + \sum_{j\in I_d} h_j^2\norm{\pdv{\tilde{\vec{v}}}{\tilde{x}_j}}_{0,\tilde{T}}^2 + h_{\tilde{T}}^2 \norm{\divtilde \tilde{\vec{v}}}_{0,\tilde{T}}^2 \right).
	\end{equation*}
	We also have 
	\begin{equation*}
		\pdv{\tilde{\vec{v}}}{\tilde{x}_j} = (\det J_T) J_T^{-1} \pdv{\vec{v}}{\vec{l}_j}, \quad \divt \vec{v}(\vec{x}) = (\det J_T)^{-1} \divtilde \tilde{\vec{v}}(\tilde{\vec{x}}), \quad h_{\tilde{T}} = \norm{J_T^{-1}}_{\infty} h_T.
	\end{equation*}
	Combining these we get
	\begin{equation*}
		\norm{\bdmi \vec{v}}_{0,T}^2 \lesssim \norm{J_T}_{\infty}^2 \norm{J_T^{-1}}_{\infty}^2 \left(\norm{\vec{v}}_{0,T}^2 + \sum_{j\in I_d} h_j^2 \norm{\pdv{\vec{v}}{\vec{l}_j}}_{0,T}^2 + h_T^2 \norm{\divt \vec{v}}_{0,T}^2 \right),
	\end{equation*}
	and with $\norm{J_T}_{\infty}, \norm{J_T^{-1}}_{\infty} \leq C$ from \autoref{lem:RVP}, where $C$ depends only on $\bar{c}$, we conclude the proof.
\end{proof}

\subsection{Stability without regular vertex property}
As mentioned before, only in three dimensions there can be elements satisfying the maximum angle condition but not a regular vertex property, so in this subsection we restrict our observations to the case $d=3$. 

As shown in \cite[Proposition 5.1]{AcostaApelDuranLombardi2011} with an example, estimates like \eqref{eq:estimate_ref_elem} can not be obtained for the Raviart-Thomas element family of any order for elements not satisfying a regular vertex property. Thus in \cite[Proposition 4.4]{AcostaApelDuranLombardi2011} a relaxed estimate for the Raviart-Thomas interpolant is proved, which we can also show for the Brezzi-Douglas-Marini case.
As in the last subsection, we start with a technical lemma.
\begin{lemma}\label{lem:bdm_zero_noRVP}
	On the reference element $\bar{T}$ with facets $e_i$, $i\in I_{d+1}$, let $\bar{f}_1 \in {L}^2(e_1)$, $\bar{f}_2 \in {L}^2(\bar{e}_2)$, $\bar{f}_3 \in {L}^2(e_3)$, where $\bar{e}_2$ is the projection of $e_2$ onto the plane $\bar{x}_2=0$, and
	\begin{align*}
		\bvec{u}(\bvec{x}) = (\bar{f}_1(\bar{x}_2,\bar{x}_3), 0, 0)^T, \quad \bvec{v}(\bvec{x}) = (0,\bar{f}_2(\bar{x}_1,\bar{x}_3),0)^T, \quad \bvec{w}(\bvec{x}) = (0,0,\bar{f}_3(\bar{x}_1,\bar{x}_2))^T.
	\end{align*}
	Then there are functions $\bar{q}_1\in P_k(e_1)$, $\bar{q}_2\in P_k(\bar{e}_2)$, $\bar{q}_3\in P_k(e_3)$, so that
	\begin{align*}
		\bbdmi\bvec{u} = (\bar{q}_1(\bar{x}_2,\bar{x}_3), 0, 0)^T, \; \bbdmi\bvec{v} = (0,\bar{q}_2(\bar{x}_1,\bar{x}_3),0)^T, \; \bbdmi\bvec{w} = (0,0,\bar{q}_3(\bar{x}_1,\bar{x}_2))^T.
	\end{align*}
\end{lemma}
\begin{proof}
	We proceed similarly to the proof of \autoref{lem:bdm_zero}, by showing that $\bar{q}_1$ is uniquely defined by the interpolation operator $\bbdmi$.
	
	The normal vectors of the facets of $\bar{T}$ are $\vec{n}_1 = (-1,0,0)^T$, $\vec{n}_2 = \frac{1}{\sqrt{2}}(1,-1,0)^T$, $\vec{n}_3 = (0,0,-1)^T$, $\vec{n}_4 = \frac{1}{\sqrt{2}}(0,1,1)^T$, see also \autoref{fig:ReferenceTetrahedron_noRVP}. By inserting the components into the relation \eqref{eq:BDM_dofs_1}, we get for $i=1$
	\begin{equation}\label{eq:bdm_zero_noRVP_1}
		\int_{e_1} \bar{f}_1 z = \int_{e_1} \bar{q}_1 z, \quad \forall z\in P_k(e_1),
	\end{equation}
	which again already defines $\bar{q}_1$ uniquely. For $i=3,4$ we again get trivial equalities, and for $i=2$ we calculate, using \eqref{eq:bdm_zero_noRVP_1},
	\begin{align*}
		\int_{e_2} (\bvec{u}-\bbdmi\bvec{u})\cdot \vec{n}_2 z &= \frac{1}{\sqrt{2}} \int_{e_2} (\bar{f}_1 - \bar{q}_1)z \\
		&= \int_0^1 \int_0^{1-\bar{x}_2} (\bar{f}_1(\bar{x}_2,\bar{x}_3) - \bar{q}_1(\bar{x}_2,\bar{x}_3))z(\bar{x}_2,\bar{x}_2,\bar{x}_3)  \dd \bar{x}_3 \dd \bar{x}_2 \\
		&= \int_{e_1} (\bar{f}_1(\bar{x}_2,\bar{x}_3) - \bar{q}_1(\bar{x}_2,\bar{x}_3))z(\bar{x}_2,\bar{x}_2,\bar{x}_3)  \dd \bar{x}_3 \dd \bar{x}_2 = 0.
	\end{align*}
	For the internal interpolation conditions \eqref{eq:BDM_dofs_2} we take an arbitrary $\vec{z}\in\vec{N}_{k-1}(\bar{T})$, and choose $Z\in P_k(\bar{T})$ so that $\pdv{Z}{\bar{x}_1} = z_1$. Let $\vec{n}=(n_1,n_2,n_3)$ be the outward normal vector on $\partial \bar{T}$. Then we calculate
	\begin{align*}
		\int_{\bar{T}} \bvec{u} \cdot \vec{z} &= \int_{\bar{T}} \bar{f}_1 z_1 = \int_{\bar{T}} \bar{f}_1 \pdv{Z}{\bar{x}_1} = \int_{\partial \bar{T}} \bar{f} Z n_1 - \int_{\bar{T}} \pdv{\bar{f}_1}{\bar{x}_1} Z \\
		&= \int_{\partial\bar{T}} \bar{q}_1 Z n_1 = \int_{\bar{T}} \bar{q}_1 \pdv{Z}{\bar{x}_1} = \int_{\bar{T}} \bar{q}_1 z_1,
	\end{align*}	
	which concludes the proof for $i=1$, the other results follow analogously.
\end{proof}

Now we show a stability estimate on the reference element $\bar{T}$.
\begin{lemma}\label{lem:Estimate_reference_element_noRVP}
	Let $\bvec{u} \in \vec{H}^1(\bar{T})$, then the estimates
	\begin{equation}\label{eq:estimate_ref_elem_norvp}
		\norm{(\bbdmi\bvec{u})_i}_{0,\bar{T}} \lesssim \norm{{\bar{u}_i}}_{1,\bar{T}} + \sum_{j\in {}_i I_d} \norm{\pdv{\bar{u}_j}{\bar{x}_j}}_{0,\bar{T}}, \quad i\in I_d,
	\end{equation}
	hold.
\end{lemma}
\begin{proof}
	Analogous to \cite[Lemma 4.3]{AcostaApelDuranLombardi2011}, we could for $i=1,3$ show estimates of the type of \eqref{eq:estimate_ref_elem} by the same steps as in the proof of  \autoref{lem:Estimate_reference_element}. These estimates lead clearly to the estimates \eqref{eq:estimate_ref_elem_norvp}. For the second component the stronger bound does not hold.
	Consider a function $\bvec{u}_* = (\bar{u}_1(0,\bar{x}_2,\bar{x}_3),0,\bar{u}_3(\bar{x}_1,\bar{x}_2, 0)^T$ and set
	\begin{equation}\label{eq:def_v}
		\bvec{v} = \bvec{u} - \bvec{u}_* = \begin{pmatrix} \bar{u}_1-\bar{u}_1(0,\bar{x}_2,\bar{x}_3) \\ \bar{u}_2 \\ \bar{u}_3-\bar{u}_3(\bar{x}_1,\bar{x}_2, 0) \end{pmatrix}.
	\end{equation}
	Then by \autoref{lem:bdm_zero_noRVP} we have $(\bbdmi\bvec{v})_2 = (\bbdmi\bvec{u})_2$ and $\divbar \bvec{v} = \divbar \bvec{u}$. Let $\bvec{v}_* = (\bar{x}_1 q_1, 0, \bar{x}_3 q_3)^T$ and 
	\begin{equation} \label{eq:def_w}
		\bvec{w} = \bvec{v} - \bvec{v}_* = \begin{pmatrix} \bar{v}_1 - \bar{x}_1 q_1 \\ \bar{v}_2 \\ \bar{v}_3 - \bar{x}_3 q_3 \end{pmatrix},
	\end{equation}
	where $q_1, q_3 \in P_{k-1}(\bar{T})$ are defined by
	\begin{align}
		\int_{\bar{T}} \bar{w}_1 z = \int_{\bar{T}} (\bar{v}_1 - \bar{x}_1 q_1) z = 0, \quad \forall z\in P_{k-1}(\bar{T}), \label{eq:q_1_norvp}\\
		\int_{\bar{T}} \bar{w}_3 z = \int_{\bar{T}} (\bar{v}_3 - \bar{x}_3 q_3) z = 0, \quad \forall z\in P_{k-1}(\bar{T}). \label{eq:q_3_norvp}
	\end{align}
	Since $\bvec{v}_* \in P_k(\bar{T})$, it holds $\bbdmi \bvec{v}_* = \bvec{v}_*$, and so $(\bbdmi \bvec{w})_1 = (\bbdmi \bvec{v})_1 = (\bbdmi \bvec{u})_1$. Now by \eqref{eq:BDM_dofs_1}, \eqref{eq:BDM_dofs_2}, and using \eqref{eq:q_1_norvp} and \eqref{eq:q_3_norvp}, $\bbdmi \bvec{w} = \bvec{p} = (\bar{p}_1,\bar{p}_2,\bar{p}_3)^T$ is defined by
	\begin{align*}
		\int_{e_1} \bar{p}_1 z &= \int_{e_1} \bar{w}_1 z = 0, &&\forall z \in P_k(e_1), \\
		\int_{e_2} (\bar{p}_1 - \bar{p}_2) z &= \int_{e_2} (\bar{w}_1 - \bar{w}_2) z \\
		&= \sqrt{2}\left(\int_{\bar{T}} \bar{w}_2 \pdv{z}{\bar{x}_2} + \int_{\bar{T}} \divbar (\bar{w}_1, \bar{w}_2,0)\right) - \int_{e_4} \bar{w}_2 z, &&\forall z \in P_k(e_2), \\
		\int_{e_3} \bar{p}_3 z &= \int_{e_1} \bar{w}_3 z = 0, &&\forall z \in P_k(e_3), \\
		\int_{e_4} (\bar{p}_2 + \bar{p}_3) z &= \int_{e_1} (\bar{w}_2 + \bar{w}_3) z  \\
		&= \sqrt{2}\left(\int_{\bar{T}} \bar{w}_2 \pdv{z}{\bar{x}_2} + \int_{\bar{T}} \divbar (0, \bar{w}_2, \bar{w}_3)\right) - \int_{e_2} \bar{w}_2 z, &&\forall z \in P_k(e_4),
	\end{align*}
	and
	\begin{align*}
		\int_{\bar{T}} \bar{p}_1 z &= \int_{\bar{T}} \bar{w}_1 z = 0, &&\forall z\in P_{k-2}(\bar{T}), \\
		\int_{\bar{T}} \bar{p}_2 z &= \int_{\bar{T}} \bar{w}_2 z, &&\forall z\in P_{k-2}(\bar{T}), \\
		\int_{\bar{T}} \bar{p}_3 z &= \int_{\bar{T}} \bar{w}_3 z = 0, &&\forall z\in P_{k-2}(\bar{T}), \\
		\int_{\bar{T}} \bvec{p} \cdot \vec{z} &= \int_{\bar{T}} \bvec{w}\cdot \vec{z} \\
		&= \int_{\bar{T}}(\bar{w}_1 z_1 + \bar{w}_2 z_2 + \bar{w}_3 z_3) = \int_{\bar{T}} \bar{w}_2 z, &&\forall \vec{z}\in \vec{S}_{k-1}(\bar{T}).
	\end{align*}
	This implies that we can estimate, analogously to \eqref{eq:bdmi_u_1},
	\begin{align*}
		\norm{(\bbdmi \bvec{u})_2}_{0,\bar{T}} = \norm{\bar{p}_2}_{0,\bar{T}} \lesssim \norm{\bar{w}_2}_{1,\bar{T}} + \norm{\divbar (0,\bar{w}_2, \bar{w}_3)}_{0,\bar{T}} + \norm{\divbar (\bar{w}_1, \bar{w}_2, 0)}_{0,\bar{T}},
	\end{align*}
	and now using \eqref{eq:def_v}, \eqref{eq:def_w} we get
	\begin{multline}\label{eq:estimate_1}
		\norm{(\bbdmi \bvec{u})_2}_{0,\bar{T}} \lesssim \\
			 \norm{\bar{u}_2}_{1,\bar{T}} + \norm{\pdv{\bar{u}_1}{\bar{x}_1}}_{0,\bar{T}} + \norm{\pdv{\bar{u}_3}{\bar{x}_3}}_{0,\bar{T}} 	+ \norm{\pdv{(\bar{x}_1 q_1)}{\bar{x}_1}}_{0,\bar{T}} + \norm{\pdv{(\bar{x}_3 q_3)}{\bar{x}_3}}_{0,\bar{T}}.
	\end{multline}
	We now have to estimate the last two terms. Observe that for all $z\in P_k(\bar{T})$
	\begin{equation*}
		0 = \int_{\bar{T}} \bar{w}_3 \pdv{z}{\bar{x}_3} = -\int_{\bar{T}} \pdv{\bar{w}_3}{\bar{x}_3} z + \int_{\partial \bar{T}} \bar{w}_3 n_3 z.
	\end{equation*}
	Choosing $z = (1-\bar{x}_2 - \bar{x}_3) \tilde{z}$, $\tilde{z} \in P_{k-1}(\bar{T})$ makes the boundary term vanish, so with the definition of $w_3$ we get
	\begin{equation*}
		\int_{\bar{T}} \pdv{(\bar{x}_3 q_3)}{\bar{x}_3} (1-\bar{x}_2 - \bar{x}_3) \tilde{z} = \int_{\bar{T}} \pdv{\bar{u}_3}{\bar{x}_3} (1-\bar{x}_2 - \bar{x}_3) \tilde{z},
	\end{equation*}
	which by similar considerations as in the proof of \autoref{lem:Estimate_reference_element} yields the estimate
	\begin{equation}\label{eq:estimate_2}
		\norm{\pdv{(\bar{x}_3 q_3)}{\bar{x}_3}}_{0,\bar{T}} \lesssim \norm{\pdv{\bar{u}_3}{\bar{x}_3}}_{0,\bar{T}}.
	\end{equation}
	Analogous steps get us
	\begin{equation}\label{eq:estimate_3}
		\norm{\pdv{(\bar{x}_1 q_1)}{\bar{x}_1}}_{0,\bar{T}} \lesssim \norm{\pdv{\bar{u}_1}{\bar{x}_1}}_{0,\bar{T}}.
	\end{equation}
	Combining \eqref{eq:estimate_1}, \eqref{eq:estimate_2} and \eqref{eq:estimate_3} yields the desired inequality.
\end{proof}

The transformation \eqref{eq:generaltransformation}, \eqref{eq:reference_family_transformation} on the reference element $\bar{T}$, yields an element $\tilde{T}$ of the reference family $\mathcal{T}_2$, see \autoref{fig:ReferenceTetrahedron_noRVP}, and gets us the following lemma.
\begin{lemma}\label{lem:transformed_element_MAC_noRVP}
	Let $\tilde{\vec{v}} \in \vec{H}^1(\tilde{T})$, where $\tilde{T} = J_{\tilde{T}} \bar{T}$. Then on the transformed element $\tilde{T}$ we have the estimate
	\begin{equation*}
		\norm{\tbdmi \tilde{\vec{v}}}_{0,\tilde{T}} \lesssim \sum_{\abs{\alpha}\leq 1} h^\alpha \norm{D^\alpha \tilde{\vec{v}}}_{0,\tilde{T}} + \sum_{i\in I_d} h_i \left( \sum_{j\in {}_i I_d} \norm{\pdv{\tilde{v}_j}{\tilde{x}_j}}_{0,\tilde{T}}\right)\lesssim \norm{\tilde{\vec{v}}}_{0,\tilde{T}} + h_{\tilde{T}} \abs{\tilde{\vec{v}}}_{1,\tilde{T}}.
	\end{equation*}
\end{lemma}
\begin{proof}
	The proof is analogous to that of \autoref{lem:transformed_element_MAC+RVP}, where instead of \autoref{lem:Estimate_reference_element} we use \autoref{lem:Estimate_reference_element_noRVP}.
\end{proof}

\begin{remark}
	It was mentioned in \cite[Remark 4.1]{AcostaApelDuranLombardi2011}, that for the Raviart-Thomas interpolation, estimates like \eqref{eq:estimate_ref_elem} and \eqref{eq:transformed_element_MAC} could be reached for the components $i\in\{1,3\}$. This is also possible for the Brezzi-Douglas-Marini interpolant, but in order to get to the better estimate on the general element, we would require this estimate for all components.
\end{remark}

\begin{theorem}\label{th:stability_MAC}
	Let $T$ be an element satisfying a maximum angle condition $\MAC(\bar{\phi})$. Then for $\vec{v} \in \vec{H}^1(T)$ the estimate
	\begin{equation}\label{eq:stability_MAC}
		\norm{\bdmi \vec{v}}_{0,T} \lesssim \norm{\vec{v}}_{0,T} + h_T \sum_{j\in I_d} \norm{\pdv{\vec{v}}{x_j}}_{0,T}
	\end{equation}
	holds, where the constant only depends on $\bar{\phi}$.
\end{theorem}
\begin{proof}
	We follow the steps from the proof of \cite[Theorem 4.1]{AcostaApelDuranLombardi2011}. The difference to the proof of \autoref{th:stability_RVP} is, that on reference family $\mathcal{T}_2$ only the weaker stability estimate holds, so similarly we only get a weaker estimate on the general element. Also instead of the directional derivatives, we now use the standard partial derivatives.
	
	Following from \autoref{lem:MAC}, there is an element $\tilde{T}\in \mathcal{T}_1\cup\mathcal{T}_2$ and an affine mapping $\tilde{\vec{x}} \mapsto J_T \tilde{\vec{x}} + \vec{x}_0$, $\norm{J_T}_\infty, \norm{J_T^{-1}}_\infty \leq C$, so that $\tilde{T}$ is mapped onto $T$. For a simpler notation we assume $\vec{x}_0 = 0$. If $\tilde{T}\in\mathcal{T}_1$, then T satisfies a regular vertex property with a constant only dependent on $\bar{\phi}$ so that \autoref{th:stability_RVP} applies, so we assume $\tilde{T}\in\mathcal{T}_2$.
	Using the definition of the Piola transform,
	\begin{equation*}
		\vec{v}(\vec{x}) = \frac{1}{\det J_T} J_T \tilde{\vec{v}}(\tilde{\vec{x}}), \quad \bdmi \vec{v} (\vec{x}) = \frac{1}{\det J_T} J_T \tbdmi \tilde{\vec{v}}(\tilde{\vec{x}}),  \quad \vec{x} = J_T \tilde{\vec{x}},
	\end{equation*} 
	\autoref{lem:transformed_element_MAC_noRVP} and changing variables, we can calculate
	\begin{align*}
		\norm{\bdmi \vec{v}}_{0,T} &\lesssim \frac{\norm{J_T}_\infty}{(\det J_T)^{\nicefrac{1}{2}}} \norm{\tbdmi \tilde{\vec{v}}}_{0,\tilde{T}} \\
		&\lesssim \frac{\norm{J_T}_\infty}{(\det J_T)^{\nicefrac{1}{2}}}  \left( (\det J_T)^{\nicefrac{1}{2}} \norm{J_T^{-1}}_\infty \norm{\vec{v}}_{0,T} + h_T \sum_{i,j\in I_d} \norm{\pdv{\tilde{v_i}}{\tilde{x}_j}}_{0,\tilde{T}} \right) \\
		&\lesssim \norm{J_T}_\infty \norm{J_T^{-1}}_\infty \left( \norm{\vec{v}}_{0,T} + h_T \norm{J_T}_\infty \sum_{i,j\in I_d} \norm{\pdv{v_i}{x_j}}_{0,T} \right). \qedhere
	\end{align*}
\end{proof}

\begin{remark}
	Instead of the directional derivatives, that were used in \autoref{th:interpolation_error_RVP}, we use the standard partial derivatives here, since handling these is easier and more natural.	
	For the anisotropic type of estimates, different paradigms are available. The one used in \autoref{th:stability_RVP} employs directional derivatives, as seen also in \cite{AcostaApelDuranLombardi2011}, while a different approach uses element independent coordinate systems and a condition relating the location of the element to the coordinate system, see \cite{Apel1999}. In either case, the element size parameters need to be defined carefully and these definitions are not necessarily the same for the two systems.
\end{remark}

\begin{example}[Necessity of the maximum angle condition]\label{subsec:necessity_MAC}
We show that estimate \eqref{eq:stability_MAC} can not be achieved without a maximum angle condition. Consider the triangle $T^*$ pictured in \autoref{fig:triangle_no_MAC}. For a decreasing parameter $h$, the interior angle at $\vec{p}_2$ gets arbitrarily close to $\pi$, thus the family of triangles does not satisfy $\MAC(\bar{\phi})$ for any $\bar{\phi}$.
\begin{figure}[t]
	\centering
	\includegraphics{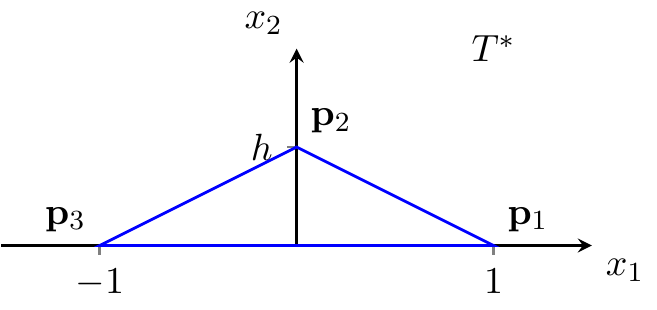}
	\caption{Family of triangles not satisfying a maximum angle condition for $h\rightarrow 0$}\label{fig:triangle_no_MAC}
\end{figure}

Let the function $\vec{v}\in L^2(T^*)$ be defined by $\vec{v}(\vec{x}) = (0, x_1^2)$. Then one can compute its Brezzi-Douglas-Marini interpolant on $T^*$ as $\bdmiorder{1} \vec{v}(\vec{x}) = (\frac{1}{2h} x_1, -\frac{1}{2h}x_2 + \frac{1}{3})^T$.
Calculating the individual terms
\begin{alignat*}{2}
	\norm{\bdmi \vec{v}}_{0,T^*} &= \sqrt{\frac{1}{24h}+\frac{h}{24}} &&\underset{h \rightarrow 0}{\longrightarrow} \infty,\\
	\norm{v_2}_{0,T^*} &= \sqrt{\frac{h}{15}} &&\underset{h \rightarrow 0}{\longrightarrow} 0,\\
	\norm{\pdv{v_2}{x_1}}_{0,T^*} &= \sqrt{\frac{2}{3}h}  &&\underset{h \rightarrow 0}{\longrightarrow} 0,
\end{alignat*}
we can see that, with an increasing aspect ratio and increasing interior angle at $\vec{p}_2$, the stability estimate from \autoref{th:stability_MAC},
\begin{equation*}
	\norm{\bdmi \vec{v}}_{0,T^*} \lesssim \norm{\vec{v}}_{0,T^*} + \sum_{i,j\in I_d} \norm{\pdv{v_i}{x_j}}_{0,T^*} = \norm{v_2}_{0,T^*} + \norm{\pdv{v_2}{x_1}}_{0,T^*},
\end{equation*}
does not hold. Hence the maximum angle condition is a necessary condition for the theorem.
\end{example}

\section{Interpolation error estimates}\label{sec:interpolation_estimates}
Before getting to the interpolation error estimates, we establish a lemma of Deny-Lions or Bramble-Hilbert type for elements satisfying a regular vertex property. As a necessary prerequisite, we state the following result without proof. It is based on a more general result from \cite{DupontScott1980}, and can be found in \cite[Lemma 2.1]{Apel1999}.
\begin{lemma}
	Let $A\subset \R^d$ be a connected set which is star-shaped with respect to a ball $B\subset A$. Let $\gamma$ be a multi-index with $\abs{\gamma} \leq k$ and $v \in H^{m+1} (A)$, $m,k\in\N$, $0\leq k \leq m+1$. Then there is a polynomial $w\in P_{m}$ so that
	\begin{equation}\label{eq:deny_lions}
		\norm{D^\gamma (v-w)}_{m+1-k,A} \lesssim \abs{D^\gamma v}_{m+1-k,A}
	\end{equation}
	holds. The constant depends only on $d$, $m$, $\diam A$, and $\diam B$. The polynomial $w$ depends only on $m$, $v$, $B$, but not on $\gamma$.
\end{lemma}

We use this lemma on a reference element, where the dependencies of the constant are clearly bounded, and then work with the same transformations as before to get to a general element. The following lemma is mainly \cite[Lemma 6.1]{AcostaApelDuranLombardi2011}.
\begin{lemma}\label{lem:interpolation}
	Let $T$ be an element satisfying $\RVP(\bar{c})$, with regular vertex $\vec{p}_{d+1}$, and $\vec{l}_i$, $h_i$ the vectors and element size parameters from \autoref{def:regular_vertex_property}. Then for $\vec{v}\in \vec{H}^{m+1}(T)$, $m\geq 0$ there is a $\vec{w} \in \vec{P}_m (T)$, so that the estimates	
	\begin{align*}
		\norm{\vec{v}-\vec{w}}_{0,T} &\lesssim \sum_{\abs{\alpha}=m+1} h^\alpha \norm{\pdv{\vec{v}}{\vec{l}^\alpha}}_{0,T},\\
		\norm{\pdv{(\vec{v}-\vec{w})}{\vec{l}_1}}_{0,T} &\lesssim \sum_{\abs{\alpha}=m} h^\alpha \norm{\frac{\partial^{m+1}\vec{v}}{\partial\vec{l}_1^{\alpha_1 + 1} \partial\vec{l}_2^{\alpha_2} \ldots \partial\vec{l}_d^{\alpha_d}}}_{0,T}
	\end{align*}
	and analogous estimates for $\pdv{(\vec{u}-\vec{w})}{\vec{l}_i}$, $i\in {}_1 I_d$, hold. Additionally the estimate
	\begin{equation*}
		\norm{\divt (\vec{v}-\vec{w})}_{0,T} \lesssim h_T^m \norm{D^m \divt \vec{v}}_{0,T}
	\end{equation*}
	holds, where $D^m f$ is the sum of the absolute values of all derivatives of order $m$ of $f$.
\end{lemma}
\begin{proof}
	We detail the proof of the first estimate, since it is not explicitly given in \cite{AcostaApelDuranLombardi2011}, and refer to the reference for the two other analogous proofs.
	
	For a simpler notation assume again $\vec{p}_{d+1} = 0$. We then know from \autoref{lem:RVP}, that there is a linear transformation with matrix $J_T$, so that an element $\tilde{T}\in\mathcal{T}_1$ gets mapped to $T$, and that $\norm{J_T}_\infty, \norm{J_T^{-1}}_\infty \leq C(\bar{c})$.
	
	Choosing $\gamma = (0,0,0)$, $k=0$ and using \eqref{eq:deny_lions} on the reference element $\hat{T}$ we get for $i\in I_d$
	\begin{align}
		\norm{\hat{{v}}_i-\hat{{w}}_i}_{0,\hat{T}} \lesssim\norm{\hat{{v}}_i-\hat{{w}}_i}_{m+1,\hat{T}} \lesssim \sum_{\abs{\alpha}=m+1} \norm{\frac{\partial^{m+1} \hat{{v}}_i}{\partial \hat{x}_1^{\alpha_1} \ldots \partial \hat{x}_d^{\alpha_d}}}_{0,\hat{T}},\label{eq:bramble_hilbert_1}
	\end{align}	
	The transformation onto the element of the reference family $\tilde{T}$ yields the functions
	\begin{align*}
		\hat{v}_i = \det(J_{\tilde{T}}) \frac{1}{h_i} \tilde{v}_i, \quad \hat{w}_i = \det(J_{\tilde{T}}) \frac{1}{h_i} \tilde{w}_i, \quad \pdv{\hat{v}_i}{\hat{x}_j} = \det(J_{\tilde{T}}) \frac{1}{h_i} \pdv{\tilde{v}_i}{\tilde{x}_j} h_j.
	\end{align*}
	Combining these relations with \eqref{eq:bramble_hilbert_1} we get
	\begin{equation*}
		\norm{\tilde{v}_i-\tilde{w}_i}_{0,\tilde{T}} \lesssim \sum_{\abs{\alpha}=m+1} h^\alpha \norm{\frac{\partial^{m+1} \tilde{v}_i}{\partial \tilde{x}_1^{\alpha_1} \ldots \partial \tilde{x}_d^{\alpha_d}}}_{0,\tilde{T}},
	\end{equation*}
	and thus
	\begin{equation*}
		\norm{\tilde{\vec{v}}-\tilde{\vec{w}}}_{0,\tilde{T}} \lesssim \sum_{\abs{\alpha}=m+1} h^\alpha \norm{\frac{\partial^{m+1} \tilde{\vec{v}}}{\partial \tilde{x}_1^{\alpha_1} \ldots \partial \tilde{x}_d^{\alpha_d}}}_{0,\tilde{T}},
	\end{equation*}
	where $\tilde{\vec{v}} \in \vec{H}^{m+1}(\tilde{T})$, $\tilde{\vec{w}}\in \vec{P}_m (\tilde{T})$ are the Piola transforms of $\vec{v}\in\vec{H}^{m+1}(T)$, $\vec{w}\in \vec{P}_m (T)$, defined by
	\begin{equation*}
		\vec{v}(\vec{x}) = \frac{1}{\det J_T} J_T \tilde{\vec{v}}(\tilde{\vec{x}}), \quad 		\vec{w}(\vec{x}) = \frac{1}{\det J_T} J_T \tilde{\vec{w}}(\tilde{\vec{x}}),\quad \vec{x} = J_T \tilde{\vec{x}}.
	\end{equation*}
	Using these definitions we get the first estimate.
\end{proof}

Finally, we can now prove the local interpolation error estimates. The proof is identical to the proof of \cite[Theorem 6.2]{AcostaApelDuranLombardi2011}.
\begin{theorem}\label{th:interpolation_error_RVP}
	Let $k\geq 1$, and let $T$ be an element satisfying $\RVP(\bar{c})$, with regular vertex $\vec{p}_{d+1}$, and $\vec{l}_i$, $h_i$ the vectors and element size parameters from \autoref{def:regular_vertex_property}. Then for $0\leq m \leq k$,  $\vec{v}\in \vec{H}^{m+1}(T)$ the estimate
	\begin{equation}\label{eq:interpolation_error_RVP}
		\norm{\vec{v}-\bdmi\vec{v}}_{0,T} \lesssim \sum_{\abs{\alpha} = m + 1} h^\alpha \norm{D^\alpha_{\vec{l}}\vec{v}}_{0,T} + h_T^{m+1} \norm{D^m\divt \vec{v}}_{0,T}
	\end{equation}
	holds, where the constant only depends on $\bar{c}$ and $k$, and $D^\alpha_{\vec{l}} = \frac{\partial^{\abs{\alpha}}}{\partial \vec{l}_1^{\alpha_1}\ldots \partial \vec{l}_d^{\alpha_d}}$.
\end{theorem}
\begin{proof}
	Due to the properties of the interpolation operator and $m\leq k$, the equality
	\begin{equation}\label{eq:interpolation_property}
		\vec{v}- \bdmi\vec{v} = \vec{v} - \vec{w} - \bdmi(\vec{v}-\vec{w})
	\end{equation}
	holds for an arbitrary function $\vec{w}\in \vec{P}_m(T)$. Now using the triangle inequality, \autoref{th:stability_RVP} and choosing $\vec{w}\in \vec{P}_m (T)$ as in \autoref{lem:interpolation}, we get
	\begin{align*}
		\norm{\vec{v}-\bdmi \vec{v}}_{0,T} &\leq \norm{\vec{v} - \vec{w}}_{0,T} +  \norm{\bdmi(\vec{v}-\vec{w})}_{0,T} \\
		&\lesssim \norm{\vec{v} - \vec{w}}_{0,T} + \sum_{i,j\in I_d} h_j \norm{\pdv{(v_i - w_i)}{\vec{l}_j}}_{0,T} + h_T\norm{\divt (\vec{v} - \vec{w})}_{0,T} \\
		&\lesssim \sum_{\abs{\alpha} = m + 1} h^\alpha \norm{D^\alpha_{\vec{l}}\vec{v}}_{0,T} + h_T^{m+1} \norm{D^m\divt \vec{v}}_{0,T}. \qedhere
	\end{align*}
\end{proof}

For elements only satisfying a maximum angle condition, we get a weaker estimate.
\begin{theorem}\label{th:interpolation_error_MAC}
	Let $k\geq 1$, and let $T$ be an element satisfying $\MAC(\bar{\phi})$. Then for $0\leq m \leq k$ and $\vec{v}\in \vec{H}^{m+1}(T)$ the estimate 
	\begin{equation}\label{eq:interpolation_error_MAC}
		\norm{\vec{v}-\bdmi\vec{v}}_{0,T} \lesssim h_T^{m+1} \norm{D^{m+1}\vec{v}}_{0,T}
	\end{equation}
	holds, where the constant only depends on $\bar{\phi}$ and $k$.
\end{theorem}
\begin{proof}
	The proof follows along the same steps as that of \autoref{th:interpolation_error_RVP}, where now the stability estimate \eqref{eq:stability_MAC} is used.
\end{proof}

\begin{example}[Weaker estimate without regular vertex property]\label{rem:weaker_estimate}
	The estimates under just a maximum angle condition, \eqref{eq:stability_MAC} and \eqref{eq:interpolation_error_MAC}, are clearly weaker than the ones under a regular vertex property, \eqref{eq:stability_RVP} and \eqref{eq:interpolation_error_RVP}. In some applications, e.g. the Stokes equations, when using appropriate finite element methods, the divergence term vanishes elementwise, so that under a regular vertex property we are left with a completely anisotropic estimate, while under a maximum angle condition and lacking a regular vertex property the individual derivatives enter the estimate with the element diameter as coefficient and we do not gain an advantage from using a small element size parameter. 
	
	An estimate of type \eqref{eq:interpolation_error_RVP} can not be achieved for elements not satisfying a regular vertex property, as we show by the following example. Consider the function $\vec{u}(\vec{x}) = ({x}_1 {x}_3, -{x}_2 {x}_3,0)^T$ on the tetrahedron $\tilde{T}$ with vertices at $\vec{p}_1 = (0,0,0)^T$, $\vec{p}_2 = (h_1,0,0)^T$, $\vec{p}_3 = (0,0,h_3)^T$, $\vec{p}_4 = (0,h_2,h_3)^T$, which is a rotated version of the tetrahedron pictured in \autoref{fig:ReferenceTetrahedron_noRVP}. Then its lowest order Brezzi-Marini-Douglas interpolant is
	\begin{equation*}
		(\bdmiorder{1} \vec{u}) (\vec{x}) = h_3 \begin{pmatrix} \frac{2}{5} x_1 \\ -\frac{3}{5} x_2 \\ -\frac{h_3}{10} + \frac{1}{5} x_3 \end{pmatrix}.
	\end{equation*}
	By directly calculating the norms, we get
	\begin{align*}
		&&&\norm{\vec{u} - \bdmiorder{1}\vec{u}}_{0,\tilde{T}} \lesssim \sum_{\abs{\alpha} = 1} h^\alpha \norm{D^\alpha \vec{u}}_{0,\tilde{T}} + h_{\tilde{T}} \norm{\divt \vec{u}}_{0,\tilde{T}} \\
		&\Leftrightarrow & &\Bigg(\sum_{i = 1}^3 \norm{{u}_i - (\bdmiorder{1}\vec{u})_i}_{0,\tilde{T}}^2\Bigg)^{\nicefrac{1}{2}} \lesssim \sum_{i = 1}^3 h_i \left( \norm{\pdv{u_1}{x_i}}_{0,\tilde{T}}^2 + \norm{\pdv{u_2}{x_i}}_{0,\tilde{T}}^2\right)^{\nicefrac{1}{2}} \\
		&\Leftrightarrow & &\left( \frac{38 h_1^2 + 38 h_2^2 + 21 h_3^2}{3150} \right)^{\nicefrac{1}{2}} \lesssim h_1 + h_2 + \left( \frac{h_1^2 + h_2^2}{3} \right)^{\nicefrac{1}{2}} \\
		&\Leftrightarrow & &21 h_3^2 \lesssim 4162(h_1^2 + h_2^2) + 6300\left(h_1 h_2 + \left(\frac{h_1^4 + h_1^2 h_2^2}{3}\right)^{\nicefrac{1}{2}} + \left(\frac{h_1^2 h_2^2 + h_2^4}{3}\right)^{\nicefrac{1}{2}}\right).
	\end{align*}
	Inspecting now both sides of the estimate, we see that for a sufficiently stretched element in the $x_3$ direction, i.e. $h_3 \gg h_1,h_2$, the inequality and thus the interpolation estimate does not hold. 

\end{example}

\section{Application to the Stokes equations}\label{sec:application}
In this final section, we give a short numerical example to illustrate that when using a finite element method with Brezzi-Douglas-Marini elements, elements with large aspect ratio do not negatively influence the convergence characteristics, and are beneficial for adequate problems.
\begin{figure}[t]
	\centering
	\subcaptionbox{Plot of the exact velocity solution for $\epsilon=0.01$\label{fig:uexact}}[0.45\textwidth]{
	\includegraphics{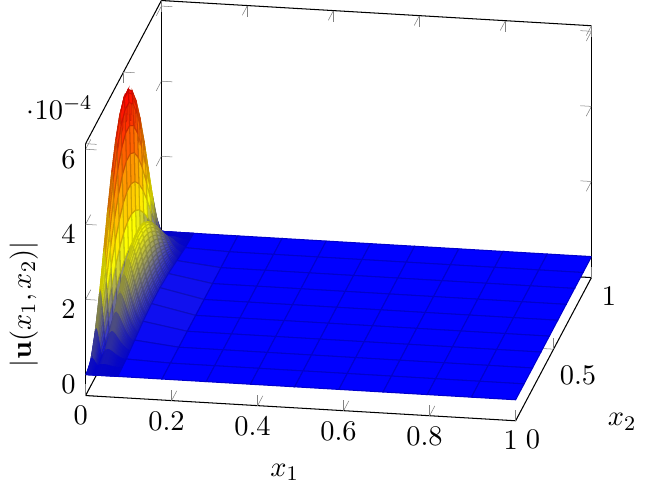}
	}\hspace{\fill}
\begin{subfigure}[t]{.45\textwidth}
	\centering
	\includegraphics{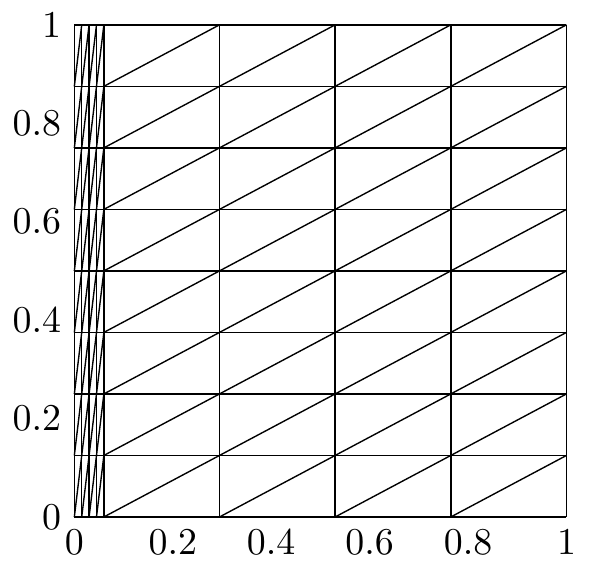}
	\caption{Shishkin type mesh for $\epsilon=0.01$, $N=2^3$, maximal aspect ratio $\sigma\approx 8.9$}
	\label{fig:mesh_ex1}
	\end{subfigure}\hspace{\fill}
	\caption{Exact velocity and example mesh used in the calculations}
	\label{fig:uexact_mesh}
\end{figure}
\begin{figure}[t]
	\centering
	\begin{subfigure}[t]{.475\textwidth}
	\centering
	\includegraphics{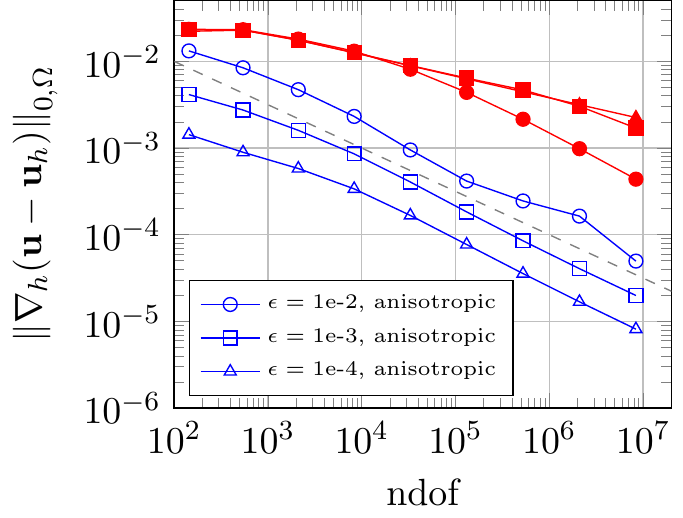}
	\end{subfigure}\hspace{\fill}
	\begin{subfigure}[t]{.475\textwidth}
	\centering
	\includegraphics{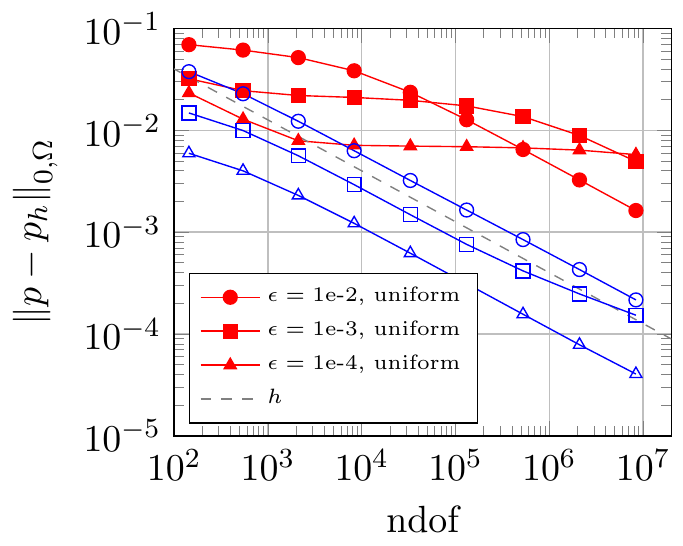}
	\end{subfigure}
	\caption{Convergence plots of the discrete velocity and pressure solutions for various values for the parameter $\epsilon$}
	\label{fig:error_ex1}
\end{figure}

The example, see also \cite{CreuseNicaiseKunert2004}, considers the steady Stokes equations on the unit square $\Omega = (0,1)^2$ in the form
\begin{align*}\label{eq:stokescont}
	-\nu \Delta\vec{u} + \nabla p &= \vec{f} \quad \text{on } \Omega,	\\
	-\nabla \cdot \vec{u} &= 0 \quad \text{on } \Omega,	\\
	\vec{u} &= \vec{g} \quad \text{on } \partial \Omega.
\end{align*}
For a complete introduction to the Stokes equations, we refer to \cite{GiraultRaviart1986}, and keep our text short.
We follow \cite{SchroederLehrenfeldLinkeLube2018} and use a mixed finite element method with lowest-order Brezzi-Douglas-Marini elements and piecewise constants for the velocity and pressure approximation, respectively. As this method is not $\vec{H}^1$-conforming, we use means from the discontinuous Galerkin framework, as described in e.g. \cite{SchroederLehrenfeldLinkeLube2018,DiPietroErn2012,CockburnKanschatSchotzau2007,CockburnKanschatSchotzauSchwab2002}. We refer to these references for the details on discontinuous Galerkin methods and their application to the Stokes equations. In particular, we employ the symmetric interior penalty formulation of the bilinear form $a_h(\vec{u}_h,\vec{v}_h)$, which is defined by, see \cite{SchroederLehrenfeldLinkeLube2018},
\begin{align*}
	a_h(\vec{u}_h, \vec{v}_h) &= \nu \int_\Omega \nabla_h \vec{u}_h : \nabla_h \vec{v}_h \dd \vec{x} \\
	&\hphantom{=} - \nu \sum_{e\in \mathcal{F}_h} \int_e \left( \avg{\nabla\vec{u}_h} \vec{n}_e \cdot \jump{\vec{v}_h} + \jump{\vec{u}_h} \cdot \avg{\nabla\vec{v}_h} \vec{n}_e - \frac{\gamma}{h_e} \jump{\vec{u}_h} \cdot \jump{\vec{v}_h}\right)\dd \vec{s},
\end{align*}
where $\avg{\cdot}$ and $\jump{\cdot}$ denote the average and jump of a function on a facet, and $\gamma$ is the jump penalization parameter.
For the calculations we choose the exact solution $(\vec{u},p)$
\begin{align*}
	&\vec{u}(\vec{x}) = 	\left(\pdv{\xi}{x_2}, -\pdv{\xi}{x_1}\right), \\
	&p(\vec{x}) = \exp(-\frac{x_1}{\epsilon}),
\end{align*}
where the stream function is defined as $\xi(\vec{x}) = x_1^2(1-x_1)^2x_2^2(1-x_2)^2\exp(-\frac{x_1}{\epsilon})$. \autoref{fig:uexact} shows a plot of the magnitude of the exact velocity for the parameter value $\epsilon = 0.01$, where the exponential boundary layer near $x_1=0$ is clearly visible. The layer has a width of $\mathcal{O}(\epsilon)$ and is also present in the pressure solution. For the calculations we use Shishkin type meshes as pictured in \autoref{fig:mesh_ex1}, which are for a parameter $N\geq 2$ constructed in the following way. For a transition point parameter $\tau \in (0,1)$ generate a grid of points $(x_1^i, x_2^j)$, 
\begin{align*}
	x_1^i &= \begin{cases}
				i\frac{2\tau}{N}, & 0\leq i\leq \frac{N}{2}, i\in \N, \\
				\tau + \left( i - \frac{N}{2}\right)\frac{2(1-\tau)}{N}, & \frac{N}{2} < i \leq N, i\in \N,
			\end{cases} \\
	x_2^j &= \frac{j}{N}, \quad 0\leq j\leq N, j \in \N.
\end{align*}
Now connect the grid points with edges to get a rectangular mesh, and subdivide each rectangle into two triangles. 
Like this we get a triangulation of $\Omega$ with $n=2N^2$ elements and an aspect ratio of $\sigma = \frac{\sqrt{1+4\tau^2}}{1+2\tau-\sqrt{1+4\tau^2}}$, see \autoref{fig:mesh_ex1}. 
For the parameters in the following calculations we choose $\nu = 1$ and $\tau = \min\{\frac{1}{2}, 3\epsilon \abs{\ln(\epsilon)}\}$, so that the anisotropic elements cover approximately three times the boundary layer width. 
Initially following \cite{SchroederLehrenfeldLinkeLube2018} for the value of the jump penalization parameter, we finally included a slight dependence on the aspect ratio and set $\gamma = 4 k^2\lceil\log(\sigma)\rceil = 4\lceil\log(\sigma)\rceil$, so that $\gamma \sim \abs{\log(\epsilon)}$ for small $\epsilon$.

From e.g. \cite{SchroederLehrenfeldLinkeLube2018,Schroeder2019,CockburnKanschatSchotzau2007} we know that for a discrete solution $(\vec{u}_h, p_h)$ of our method on a shape regular triangulation, the quantities $\norm{\nabla_h (\vec{u} - \vec{u}_h)}_{0,\Omega}$ and $\norm{p - p_h}_{0,\Omega}$ should have an order of convergence of $1$. In particular, \cite[Theorem 5.4]{Schroeder2019} states, that under the above assumptions the estimate
\begin{equation*}
	\norm{\nabla_h (\vec{u} - \vec{u}_h)}_{0,\Omega} \lesssim h \norm{\vec{u}}_{1,\Omega}
\end{equation*}
holds, and by applying our result \autoref{th:interpolation_error_MAC}, we can deduce an analogous estimate for the anisotropic triangulations. This means that under the assumption of a Lipschitz domain, and $\vec{H}^2 \times H^1$ regularity of the solution $(\vec{u},p)$ of the Stokes problem, the assumption of a shape regular triangulation can be relaxed to a maximum angle condition, while still retaining optimal convergence characteristics.

\autoref{fig:error_ex1} shows, that the convergence rate for the Shishkin type meshes is optimal, and the error is by orders of magnitude lower compared to the uniform meshes. It shows also, that the uniform meshes do not reach the optimal convergence rate, until the boundary layer is resolved sufficiently. This does not happen in our calculation for the parameter value $\epsilon = 10^{-4}$, where for the finest mesh with $8\, 392\, 704$ degrees of freedom the mesh size parameter is $h\approx\text{1.38e-3}$ and the boundary layer is $\tau\approx\text{4.0e-4}$ wide.

So in conclusion, anisotropic elements significantly reduce the error compared with the uniform meshes and the optimal convergence rate is achieved much sooner. 

\printbibliography

\end{document}